\documentclass[a4paper,11pt]{article}

\usepackage[utf8]{inputenc}
\usepackage{float}
\usepackage{amsthm,amsmath,amsfonts,amssymb}

\usepackage{soul}
\usepackage{graphicx}
\usepackage{xcolor}
\usepackage{geometry}

\usepackage[unicode]{hyperref}
\usepackage{wrapfig}
\usepackage{float}

\usepackage{listings}
\usepackage{seqsplit}

\usepackage{enumitem}
\usepackage{fullpage}
\usepackage{hyperref}
\usepackage{lineno}

\usepackage[font=small,labelfont=bf]{caption}
\usepackage[font=small,labelfont=normalfont]{subcaption}

\newtheorem{theorem}{Theorem}

\newtheorem{lemma}[theorem]{Lemma}
\newtheorem{claim}[theorem]{Claim}

\newtheorem{proposition}[theorem]{Proposition}

\newtheorem{observation}{Observation}
\newtheorem{definition}{Definition}
\newtheorem{problem}{Problem}
\newtheorem{conjecture}{Conjecture}
\newtheorem{fact}[theorem]{Fact}

\DeclareMathOperator{\RPP}{\mathbb{R}\mathcal{P}^2}

\DeclareMathOperator{\conv}{conv}
\DeclareMathOperator{\area}{area}

\DeclareMathOperator{\openu}{{\rm Open-Up}}
\DeclareMathOperator{\opend}{{\rm Open-Down}}
\DeclareMathOperator{\openud}{{\rm Open-Up\&Down}}
\DeclareMathOperator{\opendu}{{\rm Open-Down\&Up}}
\DeclareMathOperator{\opendiag}{{\rm Open-Diagonally}}

\usepackage{xspace}
\def\projective{{projective}\xspace} 
\def\affine{{affine}\xspace}

\usepackage{todonotes}

\newenvironment{romanenumerate}[1][]
{\begin{enumerate}[label=(\roman*),#1]}
	{\end{enumerate}}

\def\inst#1{$^{#1}$}

\date{\today}

\title{Erd\H{o}s--Szekeres-type problems in the real projective plane\footnote{An extended abstract of this paper appeared in the Proceedings of the 38th International Symposium on Computational Geometry (SoCG 2022).}}

\begin{document}
	\author{Martin Balko\inst{1}\thanks{M.\ Balko was supported by the grant no.~21-32817S of the Czech Science Foundation (GA\v{C}R), by the Center for Foundations of Modern Computer Science (Charles University project UNCE/SCI/004). This article is part of a project that has received funding from the European Research Council (ERC) under the European Union's Horizon 2020 research and innovation programme (grant agreement No 810115).
		} 
		\and
		Manfred Scheucher\inst{2}\thanks{M.\ Scheucher was supported by the DFG Grant SCHE~2214/1-1.}
		\and
		Pavel Valtr\inst{1}\thanks{P.\ Valtr was supported by the grant no.~21-32817S of the Czech Science Foundation (GA\v{C}R).}
	}

	\maketitle

	\begin{center}
		{\footnotesize
			\inst{1} 
			Department of Applied Mathematics, \\
			Faculty of Mathematics and Physics, Charles University, Czech Republic \\
			\texttt{balko@kam.mff.cuni.cz}
			\\\ \\
			\inst{2} 
			Institut f\"ur Mathematik, Technische Universit\"at Berlin, Germany\\
			\texttt{scheucher@math.tu-berlin.de}
			\\\ \\
		}
	\end{center}

	%%%%%%%%%%%%%%%%%%%%%%%%%%%%%%%%%%%%%%%%%%%%%%%%%%%%%%%%%%%%%
	%% Abstract
	%%%%%%%%%%%%%%%%%%%%%%%%%%%%%%%%%%%%%%%%%%%%%%%%%%%%%%%%%%%%%
	
	\begin{abstract}
		We consider point sets in the real projective plane $\RPP$ and explore variants of classical extremal problems about planar point sets in this setting, with a main focus on Erd\H{o}s--Szekeres-type problems.
		
		We provide asymptotically tight bounds for a variant of the Erd\H{o}s--Szekeres theorem about point sets in convex position in $\RPP$, which was initiated by Harborth and M\"oller in 1994.
		The notion of convex position in $\RPP$ agrees with the definition of convex sets introduced by Steinitz in 1913. 
		
		For $k \geq 3$, an \emph{(\affine) $k$-hole} in a finite set $S \subseteq \mathbb{R}^2$ is a set of $k$ points from $S$ in convex position with no point of $S$ in the interior of their convex hull. 
		After introducing a new notion of $k$-holes for points sets from $\RPP$, called \emph{projective $k$-holes}, we find arbitrarily large finite sets of points from $\RPP$ with no \projective 8-holes, providing an analogue of a classical planar construction by Horton from 1983.
		We also prove that they contain only quadratically many \projective $k$-holes for $k \leq 7$.
		On the other hand,
		we show that
		the number of $k$-holes can be substantially larger in~$\RPP$ than in $\mathbb{R}^2$
		by constructing,
		for every $k \in \{3,\dots,6\}$,
		sets of $n$ points from $\mathbb{R}^2 \subset \RPP$ with $\Omega(n^{3-3/5k})$ \projective $k$-holes and only $O(n^2)$ \affine $k$-holes.
		Last but not least, we prove several other results, for example about projective holes in random point sets in $\RPP$ and about some algorithmic aspects.
		
		The study of extremal problems about point sets in $\RPP$ opens a new area of research, which we support by posing several open problems.
	\end{abstract}

	\goodbreak
	\clearpage
	
	%%%%%%%%%%%%%%%%%%%%%%%%%%%%%%%%%%%%%%%%%%%%%%%%%%%%%%%%%%%%%
	%%%%%%%%%%%%%%%%%%%%%%%%%%%%%%%%%%%%%%%%%%%%%%%%%%%%%%%%%%%%%
	\section{Introduction}
	\label{sec:introduction}
	%%%%%%%%%%%%%%%%%%%%%%%%%%%%%%%%%%%%%%%%%%%%%%%%%%%%%%%%%%%%%
	%%%%%%%%%%%%%%%%%%%%%%%%%%%%%%%%%%%%%%%%%%%%%%%%%%%%%%%%%%%%%

	%%%%%%%%%%%%%%%%%%%%%%%%%%%%%%%%%%%%%%%%%%%%%%%%%%%%%%%%%%%%%
	\subsection{Erd\H{o}s-Szekeres-type results in the Euclidean plane}
	%%%%%%%%%%%%%%%%%%%%%%%%%%%%%%%%%%%%%%%%%%%%%%%%%%%%%%%%%%%%%
	
	Throughout the whole paper, we consider sets $S$ of points from the Euclidean plane $\mathbb{R}^2$ that are finite and in \emph{general position}, that is, no three points of $S$ lie on a common line.
	We say that a set $S$ of $k$ points in the Euclidean plane is in \emph{convex position} if $S$ forms the vertex set of a convex polygon, which we call a \emph{$k$-gon} or an \emph{\affine $k$-gon}. 
	% \manfred{we should make this properly for the full version}
	
	In 1935, Erd\H{o}s and Szekeres~\cite{ErdosSzekeres1935} showed that, for every integer $k\ge3$, there is a smallest positive integer $ES(k)$ such that every finite set of at least $ES(k)$ points in the plane in general position contains a subset of $k$ points in convex position.
	This result, known as the \emph{Erd\H{o}s--Szekeres theorem}, was one of the starting points of both discrete geometry and Ramsey theory.
	It motivated various lines of research that led to several important results as well as to many difficult open problems.
	For example, there were many efforts to determine the growth rate of the function $ES(k)$.
	Erd\H{o}s and Szekeres~\cite{ErdosSzekeres1935} showed $ES(k) \leq \binom{2k-4}{k-2}+1$ and conjectured that $ES(k)=2^{k-2}+1$ for every $k \geq 2$.
	This conjecture, known as the \emph{Erd\H{o}s--Szekeres conjecture}, was later supported by Erd\H{o}s and Szekeres~\cite{ErdosSzekeres1960}, who proved the matching lower bound $ES(k) \geq 2^{k-2}+1$. 
	The Erd\H{o}s--Szekeres conjecture was verified for $k \leq 6$~\cite{SzekeresPeters2006} (see also \cite{Maric2019,Scheucher2020}), but it is still open for $k \ge 7$. 
	In fact, Erd\H{o}s even offered \$500 reward for its solution.
	The currently best upper bound $ES(k) \le 2^{k+O(\sqrt{k\log{k}})}$ is due to Holmsen, Mojarrad, Pach, and Tardos~\cite{HolmsenMPT2020}, who improved an earlier breakthrough by Suk~\cite{Suk2017} who showed $ES(k) \le 2^{k+O(k^{2/3}\log{k})}$.
	Altogether, these estimates give, for every $k \geq 2$,
	\begin{equation}
		\label{eq-ESbounds}
		2^{k-2} +1 \le ES(k) \le 2^{k+O(\sqrt{k\log{k}})}.
	\end{equation}
	
	Several variations of the Erd\H{o}s--Szekeres theorem have been studied in the literature.
	In the 1970s, Erd\H{o}s~\cite{Erdos1978} asked whether there is a smallest positive integer $h(k)$ such that every set $S$ of at least $h(k)$ points in the plane in general position contains an \emph{(\affine) $k$-hole}, which is a 
	convex polygon spanned by a subset of $k$ points from $S$
	that does not contain any point from $S$ in its interior.
	In other words, a $k$-hole in a finite points set $S$ in the plane in general position is a $k$-gon which is \emph{empty} in~$S$,
	that is, its interior does not contain any point from~$S$.
	After Horton~\cite{Horton1983} constructed arbitrarily large point sets with no 7-hole, it took more than 20 years until Gerken~\cite{Gerken2008} and Nicolas~\cite{Nicolas2007} independently showed that every sufficiently large set of points contains a 6-hole.
	Therefore, $h(k)$ is finite if and only if $k \leq 6$.
	
	Estimating the minimum number of $k$-holes is another example of a classical Erd\H{o}s--Szekeres-type problem.
	For a fixed integer $k \geq 3$ and a positive integer $n$, let $h_k(n)$ be the minimum number of $k$-holes in any finite set of $n$ points in the plane.
	The growth rate of the function $h_k(n)$ was also studied extensively.
	Horton's result implies $h_k(n) = 0$ for $k \geq 7$.
	The minimum numbers of 3- and 4-holes are known to be quadratic in $n$, but we only have the bounds $\Omega(n \log^{4/5}n) \leq h_5(n) \leq O(n^2)$ and $\Omega(n) \leq h_6(n) \leq O(n^2)$~\cite{ABHKPSVV2020_JCTA,BaranyValtr2004} for 5- and 6-holes, respectively.
	However, it is widely conjectured that $h_5$ and $h_6$ are also both quadratic in $n$.
	
	In this paper, we consider analogous Erd\H{o}s--Szekeres-type problems in the real projective plane $\RPP$.
	We define notions of convex position, $k$-gons, and $k$-holes in $\RPP$ and study the corresponding extremal problems, providing several new results as well as numerous open problems in this new line of research.

	%%%%%%%%%%%%%%%%%%%%%%%%%%%%%%%%%%%%%%%%%%%%%%%%%%%%%%%%%%%%%
	\subsection{Convex sets in the real projective plane}
	\label{sec:prelim}
	%%%%%%%%%%%%%%%%%%%%%%%%%%%%%%%%%%%%%%%%%%%%%%%%%%%%%%%%%%%%%
	
	As in the planar case, we consider only sets $P$ of points from the real projective plane $\RPP$ that are
	finite and in \emph{general position}, that is, no three points from $P$ lie on a common projective line.
	We say that $P$ is in \emph{\projective convex position} if it is a set in convex position in some Euclidean plane $\rho \subset \RPP$.
	Recall that by removing a projective line from $\RPP$ 
	one obtains a Euclidean plane.
	Following the notation introduced by Steinitz~\cite{Steinitz1913}, we say that a subset $X$ of $\RPP$ is \emph{semiconvex} if any two points of $X$ can be joined by a line segment fully contained in~$X$. 
	The set $X$ is \emph{convex} if it is semiconvex and does not contain some projective line, that is, $X$ is contained in a plane $\rho \subset \RPP$; see also~\cite{deGrootdeVries1958}.
	A \emph{\projective convex hull} of a set $Y \subset \RPP$ is an inclusion-wise minimal convex subset of $\RPP$ containing~$Y$.
	We note that, unlike the situation in the plane, a \projective convex hull of $Y$ does not have to be determined uniquely; see Figure~\ref{fig:gonsAndHoles}.
	
	\begin{figure}[htb]
		\centering
		\includegraphics{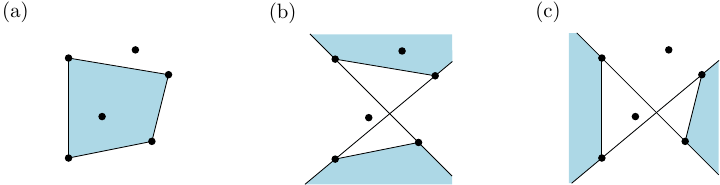}
		\caption{An example of three \projective 4-gons determined by the same subset of four points from a set $P$ of six points in $\RPP$. The \projective 4-gons in (a) and (b) are not \projective 4-holes in $P$, but the \projective 4-gon in~(c) is a \projective 4-hole in~$P$.
		}
		\label{fig:gonsAndHoles}
	\end{figure}
	
	\begin{definition}[A \projective $k$-gon]\label{def:projective-k-gon}% in $\RPP$]
		For a positive integer $k$ and a finite set $P$ of points from~$\RPP$ in general position, a \emph{\projective $k$-gon} determined by $P$ 
		is a \projective convex hull of a set $G$ of $k$ points from $P$ which contains all points of $G$ on its boundary; see Figure~\ref{fig:gonsAndHoles}.
	\end{definition}
	
	The notion ``\projective $k$-gon'' in~$\RPP$ is a natural analogue of the notion ``\affine $k$-gon'' in~$\mathbb{R}^2$, 
	since \projective $k$-gons in~$\RPP$ are exactly
	those subsets of~$\RPP$ which are affine $k$-gons in some of the planes contained in~$\RPP$.
	
	Since a \projective convex hull is not determined uniquely, a set of $k$ points in $\RPP$ can determine several \projective $k$-gons. In particular,
	it is not difficult to verify that
	\begin{romanenumerate}
		\item\label{item-3gon}
		any three points in general position in $\RPP$ determine four \projective 3-gons, 
		\item\label{item-4gon}
		any four points in general position in $\RPP$ determine three \projective 4-gons,
		\item\label{item-5gon}
		any five points in general position in $\RPP$ determine exactly one \projective 5-gon, and
		\item
		any $k \ge 6$ points in general position in $\RPP$ determine at most one \projective $k$-gon.
	\end{romanenumerate}
	
	We also introduce the following natural analogue of holes in the real projective plane.
	
	\begin{definition}[A \projective $k$-hole]
		For an integer $k \geq 3$ and a finite set $P$ of points from $\RPP$ in general position, a \emph{\projective $k$-hole} in $P$ is a \projective $k$-gon determined by points from $P$ that does not contain any point from $P$ in its interior; see Figure~\ref{fig:gonsAndHoles}.
	\end{definition}
	
	The notion of a ``\projective $k$-hole'' in~$\RPP$ is a natural analogue of the notion of an ``(\affine) $k$-hole'' in~$\mathbb{R}^2$, since \projective $k$-holes in~$\RPP$ are exactly
	those subsets of~$\RPP$ which are (\affine) $k$-holes in some of the planes contained in~$\RPP$.
	
	We note that, again, a single set of $k\in\{3,4\}$ points in general position in $\RPP$ can determine several different \projective $k$-holes.
	Also note that, if $H$ is a \projective $k$-hole in a finite set $P$ of points from $\RPP$ in general position, 
	then in every affine plane $\rho\subset \RPP$ containing~$H$, the set $H$~is an \affine $k$-hole.
	A subset of $\RPP$ is a \emph{\projective hole} in $P$ if it is a \projective $k$-hole in~$P$ for some integer $k \geq 3$.
	
	We also describe the following alternative view on \projective $k$-gons and  $k$-holes via planar point sets.
	A \emph{double chain}~\cite{hurNoyUrru99}
	is a set $S=A \cup B$ of $k$ points from $\mathbb{R}^2$ with $A=\{s_1,\dots,s_m\}$ and $B=\{s_{m+1},\dots,s_{k}\}$
	for some $m$ with $1 \leq m \leq k-1$
	such that,
	for every $i=1,\ldots,k$, 
	the line $\overline{s_is_{i+1}}$ separates 
	$A\setminus\{s_i,s_{i+1}\}$ from $B\setminus\{s_i,s_{i+1}\}$ 
	(indices modulo $k$); see Figure~\ref{fig:doubleChain}. 
	The sets $A$ and $B$ are the \emph{chains} of the double chain.
	For a line~$\ell$ not separating~$A$,
	let $H_{\ell}^A$ be the closed half-plane bounded by~$\ell$ 
	that contains~$A$ and we similarly define $H_{\ell}^B$.
	The \emph{double chain $k$-wedge} of $S$ is the union $W_A \cup W_B$
	where 
	$
	W_A = 
	\bigcap_{i=0}^{m} H^A_{\overline{s_is_{i+1}}}
	$
	and
	$
	W_B = 
	\bigcap_{i=m}^{k} H^B_{\overline{s_is_{i+1}}}
	$.

	\begin{observation}
		\label{obs-gons}
		Let $P$ be a set of $k$ points from $\RPP$ in general position
		and let $\rho\subset \RPP$ be an affine plane containing~$P$.
		A convex set $G$ in $\RPP$ is a \projective $k$-gon determined by~$P$
		if and only if, in $\rho$,
		$G$ is either 
		a convex polygon with $k$ vertices (that is, an \affine $k$-gon) 
		or a double chain $k$-wedge.\qed
	\end{observation}
	
	\begin{figure}[htb]
		\centering
		\includegraphics{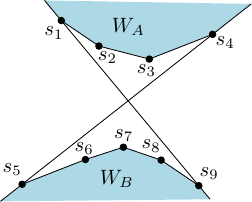}
		\caption{A double chain $S$ on 9 points and the corresponding double chain 9-wedge.
		}
		\label{fig:doubleChain}
	\end{figure}

	\begin{observation}
		\label{obs-holes}
		Let $P$ be a set of $k$ points from $\RPP$ in general position
		and let $\rho\subset \RPP$ be an affine plane containing~$P$.
		A convex set $H$ in $\RPP$ is a \projective $k$-hole in~$P$
		if and only if, in $\rho$,
		$H$ is either
		a convex polygon with $k$ vertices that is empty in~$P$ (that is, an \affine $k$-hole) 
		or 
		a double chain $k$-wedge that is empty in~$P$.\qed
	\end{observation}
	
	Convex sets in the real projective plane were considered by many authors~\cite{BrachoCalvillo1991,deGrootdeVries1958,Dekker1955,Haalmeyer1917,Kneser1921} and their study goes back more than 100 years to Steinitz~\cite{Steinitz1913}.
	Besides the article of Harborth and M\"oller \cite{HarborthMoeller1993}, where they introduced the notion of  \projective $k$-gons,
	we are not aware of any further literature on \projective $k$-gons or \projective $k$-holes.
	Thus, our goal is to conduct a first extensive study of  extremal properties of point sets in~$\RPP$.

	%%%%%%%%%%%%%%%%%%%%%%%%%%%%%%%%%%%%%%%%%%%%%%%%%%%%%%%%%%%%%
	%%%%%%%%%%%%%%%%%%%%%%%%%%%%%%%%%%%%%%%%%%%%%%%%%%%%%%%%%%%%%
	\section{Our results}
	\label{sec:results}
	%%%%%%%%%%%%%%%%%%%%%%%%%%%%%%%%%%%%%%%%%%%%%%%%%%%%%%%%%%%%%
	%%%%%%%%%%%%%%%%%%%%%%%%%%%%%%%%%%%%%%%%%%%%%%%%%%%%%%%%%%%%%
	
	First, we consider an analogue of the Erd\H{o}s--Szekeres theorem in the real projective plane. 
	For an integer $k \geq 2$, let $ES^p(k)$ be the minimum positive integer $N$ such that every set of at least $N$ points in $\RPP$ in general position contains $k$ points in \projective convex position.
	Due to Observation~\ref{obs-gons},  $ES^p(k)$ equals the minimum positive integer such that every set of at least $ES^p(k)$ points in $\mathbb{R}^2$ in general position contains either $k$ points in convex position or a double chain of size $k$.
	As already noted in \cite{HarborthMoeller1993},
	one immediately gets $ES^p(k) \leq ES(k)$.
	On the other hand,
	$ES^p(k) \geq ES(\lceil k/2\rceil)$, since the largest chain of a double chain of size $k$ has at least $\lceil k/2\rceil$ points.
	Thus, by~\eqref{eq-ESbounds}, we have $2^{\lceil k/2\rceil-2}+1\leq ES^p(k) \leq 2^{k + O(\sqrt{k\log{k}})}$ for every $k \geq 2$ and, in particular, the numbers $ES^p(k)$ are finite.
	As our first result, we prove an almost matching lower bound on $ES^p(k)$.
	
	\begin{theorem}
		\label{thm:projective_k_gon_theorem}
		There are constants $c,c'>0$ such that, for every integer $k \geq 2$,
		\[
		2^{k-c\log{k}} \leq ES^p(k) \leq 2^{k+c'\sqrt{k\log{k}}}.
		\]
	\end{theorem}
	
	The precise value of $ES^p(k)$ is known for small values of $k$.
	For $k \leq 5$, 
	all sets of $k$ points from $\RPP$ determine a \projective $k$-gon by properties~(\ref{item-3gon})--(\ref{item-5gon}) below Definition~\ref{def:projective-k-gon} and thus $ES^p(k)=k$.
	Using SAT-solver-based computations, we have also verified the value $ES^p(6) = 9$,
	which was determined by Harborth and M\"oller \cite{HarborthMoeller1993}.
	This value can also be verified with an
	exhaustive search, or by using 
	the database of order types of planar point sets~\cite{AichholzerKrasserAurenhammerOTDB,AichholzerAurenhammerKrasser2001}
	or the database of (acyclic) oriented matroids \cite{FinschiFukuda2002,FinschiDBOM}. 
	We also found sets of 17 points from~$\RPP$ with no \projective 7-gon, witnessing $ES^p(7) \geq 18$;
	see Listing~\ref{listing:coordinates17}.
	
	\begin{lstlisting}[caption={A set of 17 points from~$\RPP$ with no \projective 7-gon},label={listing:coordinates17},captionpos=b,basicstyle=\ttfamily\footnotesize]
    (45421,43357),     (24010,0),  (6531,31791), (23152,36838), (45274,43314),
    (29306,38700), (23032,36887),  (6349,32066), (23027,43265), (22863,43268), 
    (23082,43301), (22972,43303), (23085,43418),   (206,43533), 
        (0,43593), (24219,43470), (24386,44295)
	\end{lstlisting}
	
	Now, we focus on extremal problems about holes in the real projective plane.
	As our first result, we show that the existence of \projective 8-holes is not guaranteed in large point sets in~$\RPP$, proving an analogue of the result by Horton~\cite{Horton1983}.
	
	\begin{theorem}
		\label{thm:holesExistence}
		For every $n \in \mathbb{N}$, there exist sets of $n$ points from $\RPP$ in general position with no \projective 8-hole.
	\end{theorem}
	
	We recall that Theorem~\ref{thm:holesExistence} implies that there are arbitrarily large finite sets of points from~$\RPP$ in general position with no \projective $k$-holes for any $k \geq 8$.
	The proof of Theorem~\ref{thm:holesExistence} uses
	\emph{Horton sets} defined by Valtr~\cite{Valtr1992a} as a generalization
	of a construction of Horton~\cite{Horton1983} of an arbitrarily large planar point set in general position (so-called \emph{canonical Horton set}) with no $7$-hole; see Section~\ref{sec:horton_sets_proof} for the definition of Horton sets.
	Horton sets contain no \affine 7-holes in~$\mathbb{R}^2$ and we actually show that, if they are embedded in~$\RPP$, they contain no \projective 8-holes.
	Moreover, we show quadratic bounds on the number of \projective $k$-holes in Horton sets
	for $k\le 7$.
	
	\begin{theorem}
		\label{theorem:horton_sets}
		Let $H$ be a Horton set 
		of size $n$ in $\mathbb{R}^2 \subset \RPP$. 
		Then $H$ has $\Theta(n^2)$ \projective $k$-holes for every $k\le 7$.
		Moreover, if $H$ is the canonical Horton set of size $n=2^z$, then the number of \projective 3-holes in $H$ equals
		\[
		4.25 \cdot 2^{2z} + 2^z(-3z^2/2-z/2-5.5) -4z + 2 
		= 4.25n^2 -1.5n \log^2 n-\Theta(n \log n)
		.
		\]
	\end{theorem}
	
	For positive integers $k \geq 3$ and $n$, let $h_k^p(n)$ be the minimum number of \projective $k$-holes in any set of $n$ points in $\RPP$ in general position.
	Theorem~\ref{theorem:horton_sets} gives
	$h^p_k(n) \leq O(n^2)$ for every $k \leq 7$ and Theorem~\ref{thm:holesExistence} gives
	$h^p_k(n) = 0$ for every $k > 7$.
	
	In contrast to the planar case, each sufficiently large Horton set in $\RPP$ contains a \projective $7$-hole.
	We do not have examples of large point sets in $\RPP$ without \projective 7-holes, thus it is natural to ask whether there are \projective 7-holes in every sufficiently large point set in~$\RPP$.
	We believe this to be the case; see Subsection~\ref{subsec:openProblems} for more open problems.
	
	We also prove that every set of at least 7 points in $\RPP$ contains a \projective 5-hole while there are sets of 6 points in $\RPP$ with no \projective 5-hole.
	Interestingly, every set of 5 points in $\RPP$ contains a \projective 5-hole.
	This is in contrast with the situation in the plane, where we have $h_k(n) \leq h_k(n+1)$ for every $k$ and $n$, which can be seen by removing a vertex of the convex hull of a set $S$ of $n+1$ points from $\mathbb{R}^2$ with $h_k(n+1)$ \affine $k$-holes.
	
	\begin{proposition}
		\label{proposition:projective5holes}
		Every set of at least 7 points in general position in $\RPP$ contains a \projective 5-hole. 
		Also,
		$h_5^{p}(5)=1$ and $h_5^{p}(6)=0$.
	\end{proposition}
	
	The proof of Proposition~\ref{proposition:projective5holes} can be found
	in Section~\ref{sec:projective5holes_proof}.
	The following theorem shows that for some point sets the number of holes is substantially larger in~$\RPP$ than in $\mathbb{R}^2$.
	
	\begin{theorem}\label{thm:construction}
		For every $k \in \{3,\dots,6\}$ and every positive integer $n$,
		there is a set $S_k(n)$ of $n$ points in general position in $\mathbb{R}^2\subset \RPP$ such that $S_k(n)$ has 
		$O(n^2)$ \affine $k$-holes in $\mathbb{R}^2$ and 
		$\Omega(n^{3-\frac{5}{3k}})$ \projective $k$-holes. 
		
		More generally,
		for every $k \in \{3,\dots,6\}$, every real number $\alpha\in[0,k-2]$, and each positive integer $n$,
		there is a set $S^\alpha_k(n)$ of $n$ points in general position in $\mathbb{R}^2\subset\RPP$ such that $S^\alpha_k(n)$ has 
		$O(n^{2+\alpha})$ \affine $k$-holes in $\mathbb{R}^2$ and 
		$\Omega(n^{2+\beta})$ \projective $k$-holes, 
		where
		\[
		\beta:=\begin{cases} 
			1-\frac{5}{3k} + \alpha\cdot\frac{k-1}k & \ \text{if }\  0\le \alpha \le \frac{2k-5}3,\\
			(1+\alpha)\frac{k-2}{k-1} & \ \text{if }\  \frac{2k-5}3 <\alpha \le k-2.
		\end{cases}
		\]
	\end{theorem}
	
	The following result shows a significant difference between the number of holes of all sizes in the plane and in the real projective plane.
	
	\begin{theorem}\label{thm:construction2}
		For any two positive integers $n$ and $x$ with $x\le 2^{n/2}$, there is a set $S(n,x)$ of $n$ points in general position in $\mathbb{R}^2\subset\RPP$ containing at most $O(x+n^2)$ \affine holes in $\mathbb{R}^2$ and at least $\Omega(x^2)$ \projective holes. 
	\end{theorem}
	
	In general, we can show that every set $P$ of $n$ points from $\mathbb{R}^2 \subset \RPP$ contains at least quadratically many
	\projective holes which are not \affine holes in~$\mathbb{R}^2$.
	
	\begin{proposition}
		\label{proposition:manyprojective34holes}
		Let $P$ be a set of $n$ points in 
		$\mathbb{R}^2  \subset \RPP$ in general position,
		and let $h_k^p(P)$
		be the number of \projective $k$-holes in~$P$.
		Then,
		\[
		h_3^p(P) 
		\ge h_3(P) + \frac{1}{3} \binom{n}{2} \;\;\;\;\;\text{ and }\;\;\;\;\;
		h_4^p(P) 
		\ge h_4(P) + \frac{1}{2} \left(\binom{n}{2}-3n+3\right),
		\]
		where 
		$h_k(P)$
		is the number of \affine $k$-holes in $P$ in the plane $\mathbb{R}^2$.
	\end{proposition}
	
	The proof of Proposition~\ref{proposition:manyprojective34holes} 
	is in Section~\ref{sec:more_3holes_and_4holes_proof}.
	Together with the best known lower bounds on $h_3(n)$ and $h_4(n)$ by Aichholzer et al.~\cite{ABHKPSVV2020_JCTA}, the estimates from Proposition~\ref{proposition:manyprojective34holes} give
	\[
	h_3^p(n) 
	\ge \frac{7}{6}n^2 + \Omega(n \log^{2/3} n) \;\;\;\;\;\text{ and }\;\;\;\;\;
	h_4^p(n) 
	\ge \frac{3}{2}n^2 + \Omega(n \log^{3/4} n).
	\]

	We also discuss random point sets in the real projective plane and provide the following analogue to results for random point sets in the plane~\cite{BaranyFueredi1987,Valtr1995a}.
	This gives an alternative proof of the upper bound $h_3^p(n) \leq O(n^2)$.
	The proof of Theorem~\ref{theorem:random_sets} is given
	in Section~\ref{sec:random_sets}.
	
	\begin{theorem}
		\label{theorem:random_sets}
		Let $K$ be a compact convex subset in $\mathbb{R}^2$ of unit area.
		If $P$ is a set of $n$ points chosen uniformly and independently at random from $K \subset \mathbb{R}^2 \subset \RPP$,
		then the expected number of \projective 3-holes in $P$ is in $\Theta(n^2)$. 
		Moreover, the expected number of \projective holes in~$P$, which are not \affine holes in $\mathbb{R}^2$, is in $\Theta(n^2)$.
	\end{theorem}

	Last but not least, we discuss the computational complexity 
	of determining the number of $k$-gons and $k$-holes in a given point set.
	Mitchell et al.~\cite{MitchellRSW1995} gave an $O(m n^3)$ time algorithm to compute, for all $k=3,\ldots,m$, the number of $k$-gons and $k$-holes in a given set $S$ of $n$ points in the Euclidean plane.
	Their algorithm also counts $k$-islands in $O(k^2 n^4)$ time.
	Here, an \emph{(\affine) $k$-island} 
	in a finite point set $S$ in the plane in general position is 
	the convex hull of a $k$-tuple $I$ of points from $S$
	that does not contain any point from $S \setminus I$.
	Note that 
	a convex set in~$\mathbb{R}^2$ 
	is a $k$-hole in~$S$ if and only if it is a $k$-gon and a $k$-island in~$S$.
	
	Here, we consider the algorithmic aspects of the analogous problems in the real projective plane.
	By modifying the algorithm by Mitchell et al.~\cite{MitchellRSW1995}, we can efficiently
	compute the number of \projective $k$-gons, $k$-holes, and $k$-islands of a finite set in the real projective plane.
	Here, a \emph{\projective $k$-island} in a finite set $P$ of points from $\RPP$ in general position is 
	a \projective convex hull of a $k$-tuple $I$ of points from $P$
	that does not contain any point from $P \setminus I$.
	Note that, 
	similarly as in the affine case, 
	a convex set in~$\RPP$ 
	is a \projective $k$-hole in~$P$ if and only if it is a \projective $k$-gon and a \projective $k$-island in~$P$.
	
	\begin{theorem}
		\label{thm:efficient_counting}
		Let $P$ be a set of $n$ points in $\mathbb{R}^2 \subset \RPP$ in general position.
		Assuming a RAM model of computation
		which can perform arithmetic operations 
		on integers in constant time,
		we can compute the total number of \projective $k$-gons and $k$-holes in~$P$ 
		for $k=3,\ldots,m$ in $O(m n^4)$ time and $O(m n^2)$ space.
		The number of \projective $k$-islands in~$P$ for $k=3,\ldots,m$ can be computed
		in $O(m^2 n^5)$ time and $O(m^2 n^3)$ space.
	\end{theorem}

	%%%%%%%%%%%%%%%%%%%%%%%%%%%%%%%%%%%%%%%%%%%%%%%%%%%%%%%%%%%%%
	\section{Discussion}
	\label{subsec:openProblems}
	%%%%%%%%%%%%%%%%%%%%%%%%%%%%%%%%%%%%%%%%%%%%%%%%%%%%%%%%%%%%%
	
	The study of extremal questions about finite point sets in $\RPP$ suggests a wealth of interesting open problems and topics one can consider. 
	Here, we draw attention to some of them.
	
	By Theorem~\ref{thm:holesExistence}, there are arbitrarily large finite point sets in $\RPP$ that avoid $k$-holes for any $k \geq 8$.
	On the other hand, the result by Gerken~\cite{Gerken2008} and Nicolas~\cite{Nicolas2007} implies that every sufficiently large finite subset of $\RPP$ contains a \projective $k$-hole for any $k \leq 6$, as an analogous statement is true already in the affine setting.
	The existence of \projective 7-holes in sufficiently large finite subsets of $\RPP$ remains an intriguing open problem and we believe that \projective 7-holes can be always found in large points sets in $\RPP$.
	%\manfred{The proof for $h(6)\le g(9)$ in the affine setting directly translates to the projective setting and gives $h^p(6) \le g^p(9)$. Also $h(5) \le g(6)$ translates to $h^p(5) \le g^p(6)=9$; probably also other such bounds! check literature!}
	%\manfred{I have now a computer assisted proof for $h(6)\le g(9)$ in \cite{Scheucher2021_SAT_highdim}.}
	%\manfred{Koshelev's paper is a hoax, maybe mention this. not that people keep on citing his article}
	
	\begin{conjecture}
		\label{problem:7holes}
		Every sufficiently large point set in $\RPP$ contains a \projective $7$-hole.
	\end{conjecture}
	
	As we already mentioned, point sets in the plane satisfy $h_k(n) \leq h_k(n+1)$ for all $k$ and~$n$.
	By Proposition~\ref{proposition:projective5holes}, this is no longer true in the real projective plane.
	However, we do not know any other example violating this inequality except of the single case for 5-holes in $\RPP$.
	Thus, it is natural to ask the following question.
	
	\begin{problem}
		\label{problem:subsetproperty}
		Is it true that for every integer $k \geq 3$ there is $n_0=n_0(k)$ such that $h^p_k(n+1) \geq h^p_k(n)$ for every $n \geq n_0$?
	\end{problem}
	
	We have shown in Theorem~\ref{theorem:horton_sets}
	that Horton sets only contain $\Theta(n^2)$ \projective $k$-holes. 
	Since Horton sets only contain $\Theta(n^2)$ \affine $k$-islands \cite{FabilaMonroyHuemer2012}, which is asymptotically minimal,
	we wonder whether the same bound applies to \projective $k$-islands.
	
	\begin{problem}
		\label{problem:islands}
		For every fixed integer $k \geq 3$, is the minimum number of \projective $k$-islands among all sets of $n$ points from $\RPP$ in general position in $\Theta(n^2)$?
	\end{problem}
	
	By Theorem~\ref{theorem:random_sets}, the expected number of 3-holes in random sets of $n$ points from $\RPP$ is in $\Theta(n^2)$.
	In the plane, we know that the expected number of $k$-holes and $k$-islands is in $\Theta(n^2)$ for any fixed $k$~\cite{bsv2021_partI,bsv2021_partII}.
	Can analogous estimates be obtained also in the real projective plane?
	We note that the lower bound $\Omega(n^2)$ follows from the planar case.
	
	\begin{problem}
		\label{problem:quadraticrandom}
		Let $K$ be a compact convex subset in $\mathbb{R}^2$ of unit area and let $k \ge 3$.
		Is the expected number of \projective $k$-holes and $k$-islands in a set of $n$ points, which is chosen uniformly and independently at random from $K \subset \mathbb{R}^2 \subset \RPP$, in $\Theta(n^2)$?
	\end{problem}

	Besides all these Erd\H{o}s--Szekeres-type problems related to $k$-gons, $k$-holes and $k$-islands,
	many other classical problems 
	have natural analogues in the projective plane.
	In the following, we discuss the problem of \emph{crossing families}.
	Let $P$ be a finite set of points in the plane.
	For a positive integer $n$, let $T(n)$ be the largest number such that any set of $n$ points in general position in the plane determines at least $T(n)$ pairwise crossing segments.
	The problem of estimating $T(n)$ was introduced in the 1990s by Erd\H{o}s et al.~\cite{crossingFamilies} who proved $T(n) \geq \Omega(\sqrt{n})$.
	Since then it was widely conjectured that $T(n) \in \Theta(n)$.
	However, nobody has been able to improve the lower bound from \cite{crossingFamilies} until a  recent breakthrough by Pach, Rubin, and Tardos~\cite{pachRubTar21} who showed $T(n) \geq n^{1-o(1)}$.
	
	In $\RPP$, every pair of points determines a projective line that can be divided into two projective line segments.
	Given $2n$ points $p_1,\dots,p_k,q_1,\dots,q_k$ from $\RPP$, we say that they form \emph{\projective crossing family} of size $k$ if, for each $i$, we can choose a projective line segment $s_i$ between $p_i$ and $q_i$ such that for any pair $i,j$ with $1 \leq i < j \leq k$ the projective line segments $s_i$ and $s_j$ intersect.
	We can then ask about the maximum size $T^p(n)$ of a \projective crossing family in a set $P$ of $n$ points from $\RPP$.
	Note that any set of $k$ pairwise crossing segments of~$P$, 
	which live in a plane $\rho \subset \RPP$, gives a \projective crossing family of size $k$ in~$P$.
	Thus, proving a linear lower bound might be simpler for $T^p(n)$ than for~$T(n)$.
	
	\begin{problem}
		\label{problem:crossingFamily}
		Is the maximum size $T^p(n)$ of a \projective crossing family in a set of $n$ points from $\RPP$ in general position in $\Theta(n)$?
	\end{problem}

	All the notions we discussed (general position, convex position, $k$-gons, $k$-holes, $k$-islands, crossing families, and various others) 
	naturally extend 
	to higher dimensional Euclidean spaces and 
	also to higher dimensional projective spaces.
	In fact, $k$-gons and $k$-holes in higher dimensional Euclidean spaces 
	are currently quite actively studied:
	\begin{itemize}
		
		\item
		One central open problem in higher
		dimensions is to determine the largest value $H(d)$ such that every
		sufficiently large set in $\mathbb{R}^d$ contains an $H(d)$-hole.
		While  $H(2)=6$ is known,
		the gap between the upper
		and the lower bound for $H(d)$ remains huge for $d \geq 3$.
		\cite{VALTR1992b,bch20,ConlonLim2021,Scheucher2021_SAT_highdim}
		
		\item 
		For sets of $n$ points sampled independently and uniformly at random from a unit-volume convex body in $\mathbb{R}^d$, 
		the expected number of $k$-holes and $k$-islands is in $\Theta(n^d)$. \cite{bsv2021_partI,bsv2021_partII}
		
		\item
		While the $k$-gons and $k$-holes can be counted efficiently in the Euclidean plane,
		determining the size of the largest gon or hole  is NP-hard 
		already in~$\mathbb{R}^3$. 
		\cite{GiannopoulosKnauerWerner2013}

	\end{itemize}
	
	\medskip
	These analogues  
	in $\RPP$ and in high dimensional projective spaces
	are interesting by themselves, but they might also shed new light on the original problems.
	We plan to address further such analogues  and we hope to also motivate some readers for this line of research.

	%%%%%%%%%%%%%%%%%%%%%%%%%%%%%%%%%%%%%%%%%%%%%%%%%%%%%%%%%%%%%
	%%%%%%%%%%%%%%%%%%%%%%%%%%%%%%%%%%%%%%%%%%%%%%%%%%%%%%%%%%%%%
	\section{Proof of Theorem~\ref{thm:projective_k_gon_theorem}}
	\label{sec:kgons_proof}
	%%%%%%%%%%%%%%%%%%%%%%%%%%%%%%%%%%%%%%%%%%%%%%%%%%%%%%%%%%%%%
	%%%%%%%%%%%%%%%%%%%%%%%%%%%%%%%%%%%%%%%%%%%%%%%%%%%%%%%%%%%%%
	
	Here, we show, for every integer $k \geq 2$, almost matching bounds on the minimum size $ES^p(k)$ that guarantees the existence of a \projective $k$-gon in every set of at least $ES^p(k)$ points from~$\RPP$.
	More precisely, we prove that there are constants $c,c'>0$ such that
	\[
	2^{k-c\log{k}} \leq ES^p(k) \leq 2^{k+c'\sqrt{k\log{k}}}.
	\]
	The upper bound follows from~\eqref{eq-ESbounds}, thus it remains to prove the lower bound on $ES^p(k)$.
	To do so, we construct sets of $2^{k-c\log{k}}$ points in $\RPP$ with no \projective $k$-gon.
	By Observation~\ref{obs-gons}, it suffices to 
	show that $S$ contains
	no $k$ points in convex position and no double chain of size~$k$.
	To obtain such sets, we employ a recursive construction by Erd\H{o}s and Szekeres~\cite{ErdosSzekeres1935}.
	By choosing $c$ sufficiently large, we can assume $k \geq 5$.
	
	Let $X$ and $Y$ be finite sets of points in the Euclidean plane. 
	We say that \emph{$X$ lies deep below $Y$} and \emph{$Y$ lies high above $X$} if each point of $X$ lies below every line through two points of $Y$, and
	each point of $Y$ lies above every line through two points of $X$.
	For $m \geq 2$, we say that a set $C$ of $m$ points in the plane is an \emph{$m$-cup} if its points lie on the graph of a convex function and we call $C$ an \emph{$m$-cap} if its points lie on the graph of a concave function.
	
	We now construct the set $S$ inductively as follows.
	Let $a \geq 2$ and $u \geq 2$ be integers.
	If $a = 2$ or $u = 2$, we let $S_{a,u}$ be a set consisting of a single point from~$\mathbb{R}^2$ and note that $S_{a,u}$ then does not contain a 2-cap nor a 2-cup.
	If $a,u \geq 3$, then we let $S_{a,u}$ be a set obtained by placing a copy of $S_{a,u-1}$ to the left and deep below a copy of $S_{a-1,u}$.
	It follows by induction that $|S_{a,u}| = \binom{a+u-4}{a-2} = \binom{a+u-4}{u-2}$ and that $S_{a,u}$ does not contain an $a$-cap nor a $u$-cup; see~\cite{ErdosSzekeres1935}.
	Finally, we let $S = S_{\lfloor k/2\rfloor,\lfloor k/2\rfloor}$.
	Since $k \geq 5$, we have $\lfloor k/2\rfloor \geq 2$ and thus the set $S$ is well-defined.
	
	Note that $|S| = \binom{\lfloor k/2\rfloor+\lfloor k/2\rfloor-4}{\lfloor k/2\rfloor-2} \geq 2^{k-c\log{k}}$ for some constant $c>0$.
	The set $S$ does not contain $k$ points in convex position, as such a $k$-tuple contains either a $\lfloor k/2\rfloor$-cap or a $\lfloor k/2\rfloor$-cup.
	Thus, it remains to show that $S$ does not contain a double chain of size $k$.

	Suppose for contradiction that $W$ is  
	a double chain $k$-wedge with $A \cup B$ in $S$ with $A=\{s_1,\dots,s_m\}$ and $B=\{r_1,\dots,r_{k-m}\}$ for some $m$ with $1 \leq m \leq k-1$;
	using the notation from Subsection~\ref{sec:prelim}.
	We let $\ell_1$ be the line $\overline{s_1r_{k-m}}$ and $\ell_2$ be the line $\overline{s_m r_1}$.
	Let $a \leq \lfloor k/2\rfloor$ and $u \leq \lfloor k/2\rfloor$ be two numbers such that $W$ has all vertices in $S_{a,u}$ but it does not have all vertices in $S_{a-1,u}$ nor in $S_{a,u-1}$.
	Let $D$ and $U$ be the copies of $S_{a,u-1}$ and $S_{a-1,u}$, respectively, forming $S_{a,u}$.
	We can assume without loss of generality that $|\{s_1,s_m,r_1,r_{k-m}\} \cap D| \geq 2$, as the other case $|\{s_1,s_m,r_1,r_{k-m}\} \cap U| \geq 2$ is treated analogously.
	We distinguish the following two cases.

	\textbf{Case 1:}
	Assume $|\{s_1,s_m,r_1,r_{k-m}\} \cap D| = 2$.
	Then two points from $\{s_1,s_m,r_1,r_{k-m}\}$ are in $D$ and the other two points are in $U$.
	By symmetry, we can assume $s_1 \in U$.
	We distinguish the following two subcases, which are shown in
	Figure~\ref{fig:projective_gons_2}.
	Note that, since the line segments $s_1r_{k-m}$ and $s_mr_1$ cross, the cases $s_1,r_{k-m} \in U$ and $r_1,s_m \in D$ cannot occur.
	
	\begin{figure}[htb]
		\centering
		\includegraphics[page=4,width=\textwidth]{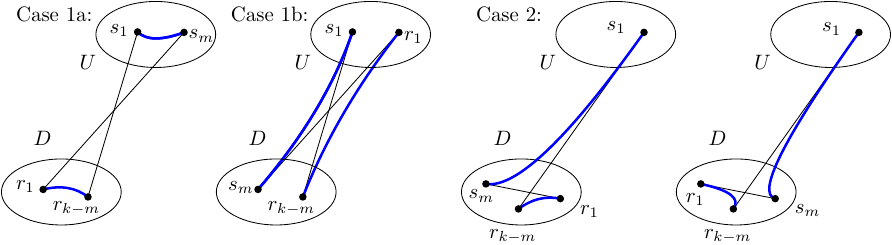} 
		\caption{The cases in the proof of Theorem~\ref{thm:projective_k_gon_theorem}.}
		\label{fig:projective_gons_2}
	\end{figure}
	
	\textbf{Case 1a:}
	Assume $s_1,s_m \in U$ and $r_1,r_{k-m} \in D$.
	We assume that $s_1$ is to the left of $s_m$, otherwise we
	reverse the order of the elements in $A$ and~$B$
	which, in particular, exchanges the roles of $s_1$ and~$s_m$.
	Since $U$ is high above $D$, the line $\overline{s_1r_{k-m}}$ is  almost vertical and separates $s_m$ from~$r_1$, where 
	$s_1$ is to the left of~$s_m$ and
	$r_1$ is to the left of~$r_{k-m}$.
	All points of $A \setminus\{s_1\}$ lie to the right of $\overline{s_1r_{k-m}}$ and to the left of $\overline{s_mr_{k-m}}$.
	Since $D$ is deep below $U$, no point of $D$ satisfies these two conditions.
	Hence all points of $A$ lie in~$U$. 
	An analogous argument shows that
	all points of~$B$ lie in~$D$.
	Since $A$ forms an $m$-cup in $U$ and $B$ forms a $(k-m)$-cap in~$D$, 
	we have $m \leq u - 1$ and $k-m \leq a-1$.
	Consequently, $k = m + (m-k) \leq a  + u - 2 \leq \lfloor k/2\rfloor + \lfloor k/2\rfloor -2 < k$, which is impossible.
	
	\textbf{Case 1b:}
	Assume $s_1,r_1 \in U$ and $s_m,r_{k-m} \in D$.
	We assume that $s_1$ is to the left of $r_1$, as otherwise we exchange the roles of $A$ and $B$
	which, in particular, exchanges the roles of $s_1$ and~$r_1$. 
	Since $U$ is high above $D$, the line $\overline{s_1r_{k-m}}$ is almost vertical and separates $s_m$ from $r_1$ and
	$s_m$ is to the left of $r_{k-m}$. 
	All points of $A \setminus \{s_1\}$ lie to the left of the almost vertical line $\overline{s_1r_{k-m}}$
	and to the right of  the almost vertical line $\overline{s_1s_m}$.
	Hence, $A \cap U = \{s_1\}$ and all points from $A \setminus \{s_1\}$ lie in $D$.
	The set $A \setminus \{s_1\}$ forms an $(m-1)$-cup in $D$ and thus $m-1 \leq u-2$.
	An analogous argument shows that $B \setminus \{r_1\}$ forms a $(k-m-1)$-cap in $U$ and thus $(k-m)-1 \leq a -2$.
	In total, we obtain $k= (m-1) + (k-m-1) + 2 \leq (u-2) + (a-2) +2 \leq \lfloor k/2 \rfloor + \lfloor k/2 \rfloor -2 < k$, which is again impossible.
	
	\textbf{Case 2:} Assume $|\{s_1,s_m,r_1,r_{k-m}\} \cap D| = 3$.
	We can assume that either $s_1$ or $s_m$ lies in~$U$, 
	as otherwise we exchange the roles of $A$ and $B$.
	Furthermore, we can assume that $s_1 \in U$, as otherwise we reverse the order of the elements in $A$ and~$B$.
	Since $U$ is high above $D$, the line $\overline{s_1r_{k-m}}$ is almost vertical and separates $r_1$ and $s_m$.
	Since all vertices of $W$ lie either to the left of the almost vertical line $\overline{s_1s_m}$ and to the right of the almost vertical line $\overline{s_1r_1}$ or to the right of $\overline{s_1s_m}$ and to the left of $\overline{s_1r_1}$,
	the point $s_1$ is the only vertex of $W$ in $U$.
	Hence, the points $S \setminus \{s_1\}$ lie in $D$ and form an $(m-1)$-cup in $D$.
	Thus, $m-1 \leq u-2$.
	The points of $B$ all lie in $D$ and form a $(k-m)$-cap in $D$.
	Thus, $k-m \leq a-1$.
	Altogether, we have $k = (m-1)+1+(k-m) \leq u-2+1+a-1 \leq \lfloor k/2\rfloor + \lfloor k/2\rfloor-2 <k$, which is impossible.
	
	\medskip
	
	Since there is no case left, we have a contradiction with the assumption that $W$ is a double chain $k$-wedge with vertices in $S$.
	This completes the proof of Theorem~\ref{thm:projective_k_gon_theorem}.

	%%%%%%%%%%%%%%%%%%%%%%%%%%%%%%%%%%%%%%%%%%%%%%%%%%%%%%%%%%%%%
	%%%%%%%%%%%%%%%%%%%%%%%%%%%%%%%%%%%%%%%%%%%%%%%%%%%%%%%%%%%%%
	\section{Proofs of Theorems~\ref{thm:holesExistence} and~\ref{theorem:horton_sets}}
	\label{sec:horton_sets_proof}
	%%%%%%%%%%%%%%%%%%%%%%%%%%%%%%%%%%%%%%%%%%%%%%%%%%%%%%%%%%%%%
	%%%%%%%%%%%%%%%%%%%%%%%%%%%%%%%%%%%%%%%%%%%%%%%%%%%%%%%%%%%%%
	
	In this section, we prove that there are arbitrarily large finite sets of points from $\RPP$ in general position with no \projective 8-hole and with only quadratically many \projective $k$-holes for every $k \leq 7$.
	We also derive the exact formula for the number of \projective 3-holes in canonical Horton sets.
	
	The construction uses so-called \emph{Horton sets} defined by Valtr~\cite{Valtr1992a}.
	Let $H$ be a set of $n$ points $p_1,\ldots,p_n$ from $\mathbb{R}^2$ with distinct $x$--coordinates $x(p_1),\dots,x(p_n)$, sorted according to increasing $x$-coordinates.
	Let $H_0$ be the set of points $p_i$ with odd $i$ and let $H_1$ be the set of points $p_i$ with even $i$.
	The set $H$ is \emph{Horton} if either $|H| \le 1$ or if $|H| \ge 2 $, $H_0$ and $H_1$ are both Horton, and $H_0$ lies deep below or high above $H_1$.
	As in Section~\ref{sec:kgons_proof}, we say that \emph{$H_0$ lies deep below $H_1$} and \emph{$H_1$ lies high above $H_0$} if every point of $H_0$ lies below every line spanned by two points of $H_1$, and
	every point of $H_1$ lies above every line spanned by two points of $H_0$.
	
	Any Horton set $H$ of at least two points is the disjoint union of some $H_0$ and~$H_1$. We denote two sets by $U(H)$ and $L(H)$ in such a way that $U(H)$ lies high above $L(H)$.
	The sets $U(H)$ and $L(H)$ are again Horton sets, thus
	the recursive use of this definition gives us the definition of the sets $U(U(H))$, $U(L(H))$, $L(U(H))$, $L(L(H))$, $U(U(U(H)))$ etc., provided that $H$ contains sufficiently many points.
	All these sets, including $H$, $U(H)$, and $L(H)$, are Horton sets according to the definition, and we call 
	them the \emph{layers} of~$H$. For example, if a Horton set has size $2^z$ for some positive integer~$z$, then it has exactly $2^{z-i}$ different layers of size $2^i$ for each $i=0,1,\dots,z$, and it has no other layers.
	
	For a nonempty subset  $A$ of $H$, 
	we define the \emph{base} of $A$ in $H$ as the smallest layer of $H$ containing $A$.
	If the base of $A$ is $H$, then we call $A$ \emph{basic in $H$}, or simply \emph{basic} if $H$ is clear from the context.
	
	As in Section~\ref{sec:kgons_proof},
	we use the terms \emph{$k$-cup} and \emph{$k$-cap}.
	A \emph{cap} is a set that is a $k$-cap for some integer $k$ and, analogously, a \emph{cup} is a set that is a $k$-cup for some $k$.
	A cap $C$ is \emph{open} in a set $S \subseteq \mathbb{R}^2$ if there is no point of $S$ below $C$,
	that is,
	for every pair of points $c_1,c_2$ from $C$,
	no point of $S$ has $x$ its coordinate between $x(c_1)$ and $x(c_2)$
	and lies below the line $\overline{c_1c_2}$.
	Analogously, a cup in $S$ is \emph{open} in $S$ if there is no point of $S$ above it.
	
	%%%%%%%%%%%%%%%%%%%%%%%%%%%%%%%%%%%%%%%%%%%%%%%%%%%%%%%%%%%%%
	\subsection{\texorpdfstring{Quadratic upper bounds on the number of $k$-holes}{Quadratic upper bounds on the number of k-holes}}\label{subsection:qudratic-upper-bound}
	%%%%%%%%%%%%%%%%%%%%%%%%%%%%%%%%%%%%%%%%%%%%%%%%%%%%%%%%%%%%%
	
	We show that any Horton set on $n$ points embedded in the real projective plane does not contain 8-holes and that $H$ has at most $O(n^2)$ projective $k$-holes for every $k \in \{3,\dots,7\}$.
	By Observation~\ref{obs-holes}, it suffices to show that any Horton set $H$ on $n$ points in the plane does not contain 8-holes nor an empty double chain 8-wedge
	and that, for every $k \in \{3,\dots,7\}$, $H$ contains only at most $O(n^2)$ $k$-holes and empty double chain $k$-wedges.
	Valtr~\cite{Valtr1992a} showed that any Horton set in the plane does not contain 7-holes and that it does not contain any open 4-cap nor an open 4-cup.
	B\'ar\'any and Valtr~\cite{BaranyValtr2004} showed that the number of $k$-holes in any Horton set of size $n$ is at most $O(n^2)$ for every $k \in \{3,\dots,6\}$.
	Thus, it suffices to estimate the number of double chain $k$-wedges in Horton sets.
	
	Let $H$ be a Horton set with $n$ points in the plane.
	We first show that the number of open caps in every Horton set $H$ with $n$ points  in the plane is at most $O(n)$ and that analogous statement is true for open cups.
	To prove this claim, it suffices to consider only open 2-caps and 3-caps, as $H$ does not contain open 4-caps.
	
	We proceed by induction on $\log_2{n}$ and show that the number $t_2(H)$ of open 2-caps equals $2n-\log_2{(n)}-2$ and that the number $t_3(H)$ of open 3-caps in $H$ equals $n-\log_2{(n)}-1$ if $n$ is a power of 2.
	Both expressions hold for $\log_2{n}=0$ and thus we assume $\log_2{n} \geq 1$.
	Let $p_1,\dots,p_n$ be the points of $H$ ordered according to increasing coordinates and let $H_0=L(H)$ and $H_1=U(H)$ be the sets that partition $H$ such that $H_0$ is deep below $H_1$.
	Every line segment $p_ip_{i+1}$ forms an open 2-cap in $H$ and there is no other open 2-cap in $H$ with points in $H_0$ and $H_1$, as there is a point of $H_1$ above any such line segment $p_ip_j$ with $j>i+1$. 
	Since no two points from $H_1$ form an open 2-cap in $H$,  we have $t_2(H_0) + n-1$ open 2-caps in $H$.
	By the induction hypothesis applied to $H_0$, it follows that $t_2(H)=2n-\log_2{(n)}-2$.
	
	To determine the number of open 3-caps in $H$, 
	note that every triple $p_ip_{i+1}p_{i+2}$ with odd $i$ if $H_0$ contains odd points of $H$ and even $i$ if $H_0$ contains even points of $H$ forms an open 3-cap in $H$.
	In fact, there is no other open 3-cap in~$H$ with a point in $H_0$ and also in~$H_1$, as there is a point of $H_1$ above any such line segment $p_ip_j$ with $j>i+1$. 
	Since no three points in $H_1$ form an open 3-cap in $H$, we obtain $t_3(H_0) + n/2-1$ open 3-caps in $H$.
	The induction hypothesis for $H_0$ then gives $t_3(H)=n-\log_2{(n)}-1$.
	
	If $n$ is not a power of two, we consider a Horton set $H'$ of size $m$ instead, where $m$ is as the smallest power of 2 larger than~$n$, and denote its leftmost $n$ points by~$H''$. 
	Since $H''$ is also a Horton set of $n$ points and 
	contains the same open caps as~$H$,
	we obtain $t_2(H) \le t_2(H') < 4n$ and $t_3(H) \le t_3(H') < 2n$. 
	Overall, the number of open caps in $H$ is at most~$O(n)$.
	With an analogous argument we obtain the same upper bound on the number of open 2- and 3-cups~in~$H$.

	We now proceed with the proof by induction on $n$.
	For some integer $k \geq 3$, let $W \subseteq H$ be a double chain $k$-wedge that is empty in $H$.
	We will show that $k \leq 7$ and, for $k \in \{3,\dots,7\}$, we estimate the number of such double chain $k$-wedges with points in $H_0$ and $H_1$ by $O(n^2)$.
	Clearly, the claims about the double chain $k$-wedges are true in any Horton set with one or two points, so we assume $n \geq 3$.
	
	If $W$ is contained in $H_0$ or in $H_1$, then $k \leq 7$ by the induction hypothesis.
	Thus, we assume that $W$ contains a point from $H_0$ and also from $H_1$.
	Using the notation from Subsection~\ref{sec:prelim}, let $A= \{s_1,\dots,s_m\}$ and $B = \{r_1,\dots,r_{k-m}\}$ be the vertices of $W$ such that $A$ and $B$ lie on opposite sides of the double chain. 
	We have $1 \leq m \leq k-1$ and, by symmetry, we can assume without loss of generality that $m \leq k-m$. 
	
	\textbf{Case 1:}
	Assume $m=1$. 
	By symmetry, we can assume that $s_1 \in H_1$.
	Moreover, we can assume $r_{k-1} \in H_0$ because if both points $r_1$ and $r_{k-1}$ lie in $H_1$, then $W \subseteq H_1$, which is impossible as $W \cap H_0$ is nonempty.
	
	\textbf{Case 1a:}
	Assume $r_1 \in H_0$; see Figure~\ref{fig:horton_holes_2}.
	No point $r_i$ lies in~$H_1$ as 
	otherwise the line $\overline{s_1 r_i}$ separates $r_1$ from $r_{k-1}$, which is impossible as $H_1$ is high above~$H_0$.
	The points from $B$ form a $(k-1)$-cap $C$ in~$H_0$.
	The cap $C$ is is open in $H$, as $W$ is empty in~$H$.
	Since there are no open 4-caps in $H$, we get $k \leq 4$.
	Also, since the number of open caps in $H$ is at most $O(n)$ and there are at most $n$ choices for $s_1$, 
	there are at most $O(n^2)$ choices for $W$ in this case.

	\textbf{Case 1b:}
	Assume $r_1 \in H_1$; see Figure~\ref{fig:horton_holes_2}.
	Then $r_{k-1}$ is the only vertex of $W$ in $H_0$, 
	as otherwise for any other point $r_i$ in $H_0$ 
	the line $\overline{r_{k-1} r_i}$ 
	separates $s_1$ from $r_1$, which is impossible as $H_0$ is deep below $H_1$.
	Moreover, $r_{k-1}$ is the leftmost or rightmost point of $H_0$, as otherwise there is a point from $H_0$ below $\overline{s_1r_1}$ and above $\overline{s_1r_{k-1}}$ and then $W$ is not empty in $H$. 
	The points of $B \setminus \{r_{k-1}\}$ form a $(k-2)$-cap $C$ in $H_1$.
	Since $W$ is empty in $H$, the cap $C$ is open in $H_1$.
	As there are no open 4-cups in $H_1$, we get $k \leq 5$.
	There are $O(n)$ open caps in $H_1$, at most two choices for $r_{k-1}$, and at most $O(n)$ choices for $s_1$.
	Altogether, we have at most $O(n^2)$ choices for $W$ in this case.
	
	\begin{figure}[htb]
		\centering
		\includegraphics[page=7,width=\textwidth]{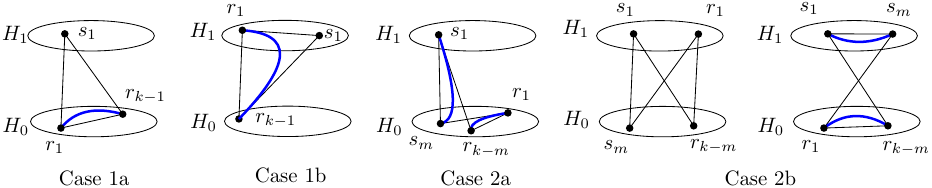}
		\caption{An illustration of the cases in the proofs of Theorems~\ref{thm:holesExistence} and~\ref{theorem:horton_sets}.} 
		\label{fig:horton_holes_2}
	\end{figure}

	\textbf{Case 2:}
	Assume $m \ge 2$.
	By symmetry, we can assume that at least two points from $\{s_1,s_m,r_1,r_{k-m}\}$ are in $H_0$ and that $s_1 \in H_1$.
	We distinguish the following two subcases.
	
	\textbf{Case 2a:}
	Assume $|\{s_1,s_m,r_1,r_{k-m}\}\cap H_0|=3$, that is, all three points $s_m,r_1,r_{k-m}$ lie in~$H_0$; see Figure~\ref{fig:horton_holes_2}.
	By symmetry, we can assume that 
	$s_m$ lies to the left of $r_1$ as the case when $s_m$ lies to the right of $r_1$ is analogous.
	Since the two line segments $s_1r_{k-m}$ and $s_mr_1$ cross, the point $r_{k-m}$ lies between $s_1$ and~$r_1$. 
	Next, we observe that $s_1$ is the leftmost point of~$H_1$, 
	as otherwise there would be a point of $H_1$ to the left of the line $\overline{s_1r_{k-m}}$ and above $\overline{s_mr_1}$, and then $W$ would not empty in~$H$.
	Since $H_0$ is deep below $H_1$,
	all points of $A \setminus \{s_1\}$ lie in~$H_0$. 
	Since all points of $B$ are contained in the triangle spanned by the points $r_1, r_{k-m},s_m$, which all lie in $H_0$, all points of $B$ lie in~$H_0$.
	The points of $A \setminus \{s_1\}$ form an $(m-1)$-cup $C$ in~$H_0$
	and 
	the points of $B$ form a $(k-m)$-cap $C'$ in~$H_0$.
	The cup $C$ and cap $C'$ are open in~$H_0$, as $W$ is empty in~$H$.
	Since $H$ does not contain open 4-cups nor open 4-caps, we get $k = (m-1) + 1 + (k-m) \leq 3 + 1 + 3 \leq 7$.
	Furthermore, there are $O(n)$ choices for $C$ and $C'$ and two choices for $s_1$, and thus we again have at most $O(n^2)$ choices for~$W$.

	\textbf{Case 2b:}
	Assume that $|\{s_1,s_m,r_1,r_{k-m}\}\cap H_0|=2$.
	Note that, since the two segments $s_1r_{k-m}$ and $r_1s_m$ cross,
	the case $s_1,r_{k-m} \in H_1$ and $r_1,s_m \in H_0$ cannot occur.
	
	First, we consider the case $s_1,r_1 \in H_1$ and $s_m,r_{k-m} \in H_0$; see Figure~\ref{fig:horton_holes_2}.
	Analogously as before, we see that each of the points $s_1,r_1$ is the leftmost or the rightmost point of~$H_1$, and each of the points $s_m,r_{k-m}$ is the leftmost or the rightmost point of~$H_0$.
	Since $H_0$ is deep below $H_1$, we have $k=4$ and there is at most one choice for $W$ in this case.
	
	Now, we assume that $s_1,s_m \in H_1$ and $r_1,r_{k-m} \in H_0$; see Figure~\ref{fig:horton_holes_2}.
	Since each point of $A$ lies on or above the lines $\overline{s_1r_{k-m}}$ and $\overline{s_mr_1}$, all points of $A$ lie in~$H_1$.
	Analogously, all points of $B$ lie in~$H_0$.
	The points from $A$ form an $m$-cup $C$ in $H_1$ and the points from $B$ form a $(k-m)$-cap in~$H_1$.
	Since $W$ is empty in $H$, the cup $C$ and the cap $C'$ are open in~$H$.
	Since there are no open 4-caps or 4-cups in~$H$,
	we obtain $k=m+(k-m) \leq 3+3\leq 6$.
	Furthermore, there are at most $O(n)$ choices for $C$ and $C'$ and thus, again, we have at most $O(n^2)$ choices for~$W$.
	
	\medskip
	
	Altogether, we see that $H$ contains no double chain 8-wedge that is empty in $H$ and that has points in $H_0$ and $H_1$.
	By the induction hypothesis, the sets $H_0$ and $H_1$ do not contain any double chain 8-wedge that is empty in~$H_0$ and in $H_1$, respectively.
	Since every double chain 8-wedge that is contained in $H_i$ and is empty in $H$ is also empty in $H_i$ for every $i \in \{0,1\}$, we see that there is no double chain 8-wedge in~$H$ that is empty in~$H$.
	This completes the proof of Theorem~\ref{thm:holesExistence}.
	
	Let $k \in \{3,\dots,7\}$.
	For the quadratic upper bounds, we have shown that there is a constant $c$ such that $H$ contains at most $cn^2$ double chain $k$-wedges that are empty in~$H$ and that have points in $H_0$ and~$H_1$.
	Altogether, the number $w_k(H)$ of empty double chain $k$-wedges in $H$ satisfies $w_k(H) \leq w_k(H_0) + w_k(H_1) + cn^2$.
	Solving this linear recurrence with the initial condition $w_k(H')=0$ for any set $H'$ with $|H'| =1$ gives $w_k(H) \leq O(n^2)$.
	This completes the proof of the first part of Theorem~\ref{theorem:horton_sets}.

	%%%%%%%%%%%%%%%%%%%%%%%%%%%%%%%%%%%%%%%%%%%%%%%%%%%%%%%%%%%%%
	\subsection{Holes in the canonical Horton set}
	\label{subsec:holes_perfect_horton_set}
	%%%%%%%%%%%%%%%%%%%%%%%%%%%%%%%%%%%%%%%%%%%%%%%%%%%%%%%%%%%%%
	
	In this subsection, we deal with point sets in ${\mathbb{R}}^2\subset\RPP$.
	We explain all the constructions in the plane, and therefore, for the sake of simplicity, we use the notions \emph{affine hole} and \emph{projective hole} for an (affine) hole in the plane and for the intersection of a projective hole with ${\mathbb{R}}^2$, respectively. Thus, a non-affine projective hole is an empty double chain wedge.

	We compute the exact number of $3$-holes of so-called canonical Horton set of size $n=2^z$, where $z$ is an integer. 
	The \emph{canonical Horton set} is a Horton set such that for each its layer $X$, the leftmost point of $X$ lies in $L(X)$.
	Up to the order type equivalence, the canonical Horton set of a given size is unique.
	Horton~\cite{Horton1983} constructed canonical Horton sets
	of size $2^z$, where $z$ is a positive integer, to show the existence of arbitrarily large planar point sets with no $7$-hole. 
	We denote the canonical Horton set of size $n$ by $CH(n)$.
	
	For a finite set $H$ of points in $\mathbb{R}^2$, we denote by ${\cal A}(H)$ the arrangement of lines determined by $H$. If the points of $H$ have distinct $x$-coordinates, we denote
	by ${\cal A}(H)^+$ and ${\cal A}(H)^-$ the two ``opposite'' unbounded cells of ${\cal A}(H)$ which lie above (below, respectively) all lines of ${\cal A}(H)$.
	
	%%%%%%%%%%%%%%%%%%%%%%%%%%%%%%%%%%

	\begin{lemma}\label{lemma:regions}
		If $X$ is a layer of a Horton set $H$, then all points of $H\setminus X$
		lie in ${\cal A}(X)^+\cup{\cal A}(X)^-$.
	\end{lemma}
	
	\begin{proof}
		The lemma is trivially true for $X=H$. Thus, it suffices to prove that
		if the lemma holds for some layer $X$ of $H$ then it also holds for the two layers $U(X)$, $L(X)$ of $H$. We prove it for the layer $U(X)$; the proof for $L(X)$ is analogous. We need to prove that the points of $H\setminus U(X)$
		lie in ${\cal A}(U(X))^+\cup{\cal A}(U(X))^-$. We have
		$H\setminus U(X)= L(X)\cup(H\setminus X)$. The points of $L(X)$ lie in
		${\cal A}(U(X))^-$ by the definition of Horton sets.
		The points of $H\setminus X$ lie in $ {\cal A}(X)^+\cup{\cal A}(X)^-  \subseteq  {\cal A}(U(X))^+\cup{\cal A}(U(X))^-$
		by the assumption that the lemma holds for the layer $X$ of $H$. 
	\end{proof}
	
	Observe that a projective hole $P$ in a Horton set either intersects every vertical line in a bounded lie segment or the intersection is unbounded for every vertical line.
	In the first case, we call the projective hole $P$ \emph{skew} and in the other case we call $P$ \emph{non-skew}.
	
	\begin{proposition}\label{prop:skew}
		Let $X$ be a layer in a Horton set $H$, and let $X\neq H$. A projective hole of $X$ is a projective hole of $H$ if and only if it is skew.
	\end{proposition}
	
	\begin{proof}
		Suppose first that a projective hole $A$ of $X$ is skew.
		Then the interior of $A$ is disjoint from ${\cal A}(X)^+\cup{\cal A}(X)^-$
		and therefore it is also disjoint from ${\cal A}(H)^+\cup{\cal A}(H)^-$
		($\subseteq {\cal A}(X)^+\cup{\cal A}(X)^-$).
		Thus, due to Lemma~\ref{lemma:regions}, the interior of $A$ contains no points of $H\setminus X$. Consequently, $A$ is a projective hole of the whole set $H$.
		
		Suppose now that a projective hole $A$ of $X$ is non-skew. Then the region 
		${\cal A}(X)^+\cup{\cal A}(X)^-$ (containing all points of $H\setminus X$
		due to Lemma~\ref{lemma:regions}) is contained in the interior of $A$.
		Since $H\setminus X\neq \emptyset$, it follows that $A$ is not a projective hole of $X$.
	\end{proof}
	
	Below, we distinguish type 1a and type 1b 3-holes according to Cases 1a and 1b in
	Subsection~\ref{subsection:qudratic-upper-bound} and we estimate the number of 3-holes of each type in the following three claims.
	
	%%%%%%%%%%%%%%%%%%%%%%%%%%%%%%%%%%

	%%%%%%%%%%%%%%%%%%%%%%%%%%%%%%%%%%
	\begin{claim}[affine $3$-holes]
		\label{claim:affine-3-holes}
		Any Horton set of size $n=2^z$ contains 
		\[
		2^{2z+1}- \left(\frac12 z^2 + \frac32 z +2\right)2^z \ \ \ \ \ \ \ \ \ \ \ \ \ \ \ \left(=2n^2-  
		\left(\frac12 \log^2_2n + \frac32\log_2n +2\right) n \right)
		\]
		$3$-holes in $\mathbb{R}^2$.
	\end{claim}
	
	%%%%%%%%%%%%%%%%%%%%%%%%%%%%%%%%%%
	\begin{claim}[type 1a $3$-holes]
		\label{claim:type1a-3-holes}
		Any Horton set of size $n=2^z$ embedded in $\RPP$ contains 
		\[
		2^{2z+1} - (z^2-z+6)2^z + 4
		\ \ \ \ \ \ \ \ \ \ \ \ \ \ \ \left(=2n^2-  
		\left(\log_2^2n -\log_2n +6\right) n  +4\right)
		\]
		type 1a $3$-holes.
	\end{claim}
	
	%%%%%%%%%%%%%%%%%%%%%%%%%%%%%%%%%%
	\begin{claim}[type 1b $3$-holes]
		\label{claim:type1b-3-holes}
		The canonical Horton set of size $n=2^z$ embedded in $\RPP$ contains 
		\[
		2^{2z-2} + \frac52 2^z -4z -2 
		\ \ \ \ \ \ \ \ \ \ \ \ \ \ \ \left(=\frac14 n^2 +  \frac52 n
		-4\log_2n - 2 \right)
		\]
		type 1b $3$-holes.
	\end{claim}
	
	Claims~\ref{claim:affine-3-holes}--\ref{claim:type1b-3-holes} immediately give the following theorem, since every projetive $3$-hole in a point set in $\mathbb{R}^2\subset\RPP$ is either an affine $3$-hole or a type 1a or type 1b projective $3$-hole.
	
	%%%%%%%%%%%%%%%%%%%%%%%%%%%%%%%%%%
	\begin{theorem}
		The canonical Horton set of size $n=2^z$ embedded in $\RPP$ contains 
		\[
		2^{2z+2} + 2^{2z-2} - \left(\frac32 z^2 + \frac12 z + \frac{11}2 \right) 2^z -4z + 2 
		\]
		\projective $3$-holes.
	\end{theorem}

	\begin{proof}[Proof of Claim~\ref{claim:affine-3-holes}]
		Let $H$ be a Horton set of size $2^z$.
		First, we determine the number of basic affine $3$-holes. 
		The affine $3$-holes of $H$ with one point in $U=U(H)$ and two points in $L=L(H)$ are formed exactly by triples consisting of a point of $U$ and an open 2-cup in $L$. 
		Thus, there are $|U|\cdot\openu_2(L)$ of them, where $\openu_2(X)$ denotes the number of open 2-cups in a set~$X$.
		Similarly, there are $|L|\cdot\opend_2(U)$ affine $3$-holes of $H$ with one point in $L$ and two points in $U$, where $\opend_2(X)$ denotes the number of open 2-caps in a set $X$.
		It follows that the number of basic affine $3$-holes in $H$ is
		$|U|\cdot\openu_2(L)+|L|\cdot\opend_2(U)=2\cdot 2^{z-1}\cdot\left(2^z-z-1\right)
		=2^{2z}-z 2^z - 2^z$.
		Similarly, every layer $X$ of $H$ of size $2^i$ has $ 2^{2i}-i 2^i - 2^i $
		affine $3$-holes with base $X$.
		Since $H$ has exactly $ 2^{z-i} $ layers of size $2^i$ for each $i=2,\dots,z$, we have
		\begin{align*}
			h_3(H)
			&=\sum_{i=2}^z  2^{z-i} \left(2^{2i}-i 2^i - 2^i\right)
			=\sum_{i=2}^z  \left(2^{z+i}-i 2^z - 2^z\right)
			\\
			&=\left( 2^{2z+1}-2^{z+2} \right)-\left(\frac{z^2+z}2 -1 \right)2^z
			-(z-1)2^z
			=2^{2z+1}- \left(\frac{z^2}2 + \frac{3z}2 +2\right)2^z.
		\end{align*}
	\end{proof}
	
	\begin{proof}[Proof of Claim~\ref{claim:type1a-3-holes}]
		Let $H$ be a Horton set of size $2^z$.
		We first compute the number of basic type 1a projective $3$-holes in $H$.
		Type 1a $3$-holes with one point in $U$ and two points in $L$ are formed exactly by triples consisting of a point of $U$ and an open 2-cap in $L$. Their number is
		$|U|\cdot\opend_2(L)=2^{z-1}\cdot\left(2^z-z-1\right)$. 
		Similarly, type 1a $3$-holes with two points in $U$ and one point in $L$ are formed exactly by triples consisting of a point of $L$ and an open 2-cup in $U$.
		Their number is $|L|\cdot\openu_2(U)=2^{z-1}\cdot\left(2^z-z-1\right)$.
		Hence the number of basic type 1a projective $3$-holes in $H$ is
		$2\cdot 2^{z-1}\cdot\left(2^z-z-1\right)
		=2^{2z}-z 2^z - 2^z$.
		
		Consider now a layer $X$ of $H$ of size $2^i$ where $i\ne z$. We now compute the number of basic type 1a projective $3$-holes in $X$. First, we compute the number of these holes with one point in $U(X)$ and two points in $L(X)$. 
		%We use the notation as in the left picture in Figure~\ref{fig:horton_holes_2}.
		Let $\{\alpha\}$ and $\{\beta, \gamma\}$ be the two chains of our 3-hole.
		By Proposition~\ref{prop:skew}, the point $\alpha$ does not lie between $\beta$ and $\gamma$ in the left-to-right order.
		Note that a 2-cap with points $\beta,\gamma$ of $L(X)$ is open if an only if it is a consecutive (in the left-to-right order) pair of points in one of the
		sets $L(X),L(L(X))=L^2(X),\dots,L^{z-1}(X)$. 
		Let $j\in\{1,\dots,i-1\}$, and let $\{x,y\}$ be a pair of consecutive (in the left-to-right order) points of~$L^j(X)$. Exactly $2^{j-1}$ points of $U(X)$ 
		lie between $x$ and $y$ (in the left-to-right order). The pair $\{x,y\}$
		forms a skew type 1a projective hole with each of the remaining
		$2^{i-1}-2^{j-1}$ points of $U(X)$.
		Since there are $|L^j(X)|-1=2^{i-j}-1$ consecutive pairs $\{x,y\}$ in $L^j(X)$,
		the number of basic skew type 1a projective $3$-holes of $X$ containing one point in $U(X)$ and two points in $L(X)$ is
		\[
		\sum_{j=1}^{i-1} (2^{i-j}-1)\cdot (2^{i-1}-2^{j-1}).
		\]
		By symmetry, we have the same number of 
		basic type 1a projective $3$-holes of $X$ containing two points in $U(X)$ and one point in $L(X)$. Hence the number of basic skew type 1a projective $3$-holes in $X$ is
		\[
		2\cdot\sum_{j=1}^{i-1} (2^{i-j}-1)\cdot (2^{i-1}-2^{j-1})
		=\sum_{j=1}^{i-1}\left( 2^{2i-j} + 2^j -2\cdot 2^i \right)
		\]
		\[
		=\left(2^{2i}-2^{i+1}\right) + (2^i-2) -(i-1)2\cdot2^i
		=2^{2i}-(2i-1)2^i-2.
		\]
		
		Summing up over all layers, we conclude that the total number of type 1a projective $3$-holes is
		\[
		(2^{2z}-z 2^z - 2^z) + \sum_{i=2}^{z-1}  2^{z-i}\left( 2^{2i}-(2i-1)2^i-2\right)
		\]
		\[
		=(2^{2z}-z 2^z - 2^z) + \sum_{i=2}^{z-1} \left(2^{z+i} -(2i-1)2^z - 2\cdot  2^{z-i} \right)
		\]
		\[
		=(2^{2z}-z 2^z - 2^z) + (2^{2z} - 2^{z+2} ) - (z^2-2z)2^z - (2^z-4)
		\]
		\[
		= 2^{2z+1} + (-z^2+z-6)2^z + 4.
		\]
	\end{proof}
	
	Before proving Claim~\ref{claim:type1b-3-holes} we give two definitions and two auxiliary claims.
		Let $x$ and $y$ be two points in a Horton set $H$ and let $x$ lie to the left of $y$. We say that the pair $\{x,y\}$ is \emph{open up}\&\emph{down} if no point of $H$ to the left of $x$ lies above the line $\overline{xy}$ and no point of $H$ to the right of $y$ lies below the line $\overline{xy}$; see part~(a) of Figure~\ref{fig-forbidden}.
	In the context of the (potentially) open up\&down pair $\{x,y\}$, the unbounded region bounded from below by the line $\overline{xy}$ and from right by the vertical line trough $x$
	is called \emph{the left forbidden region}, and the region bounded from above by the line $\overline{xy}$ and from left by the vertical line trough $y$ is called \emph{the right forbidden region}. 
	Similarly, we say that the pair $\{x,y\}$ is \emph{open down}\&\emph{up} if no point of $H$ to the left of $x$ lies below the line $\overline{xy}$ and no point of $H$ to the right of $y$ lies above the line $\overline{xy}$; see part~(b) of Figure~\ref{fig-forbidden}.
	In the context of the (potentially) open down\&up pair $\{x,y\}$, the unbounded region bounded from above by the line $\overline{xy}$ and from right by the vertical line trough $x$
	is called \emph{the left forbidden region}, and the region bounded from below by the line $\overline{xy}$ and from left by the vertical line trough $y$ is called \emph{the right forbidden region}.
	
	\begin{figure}[htb]
		\centering
		\includegraphics{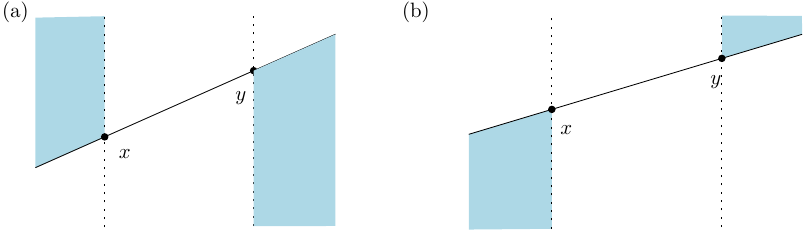} 
		\caption{An example of a pair $\{x,y\}$ that is (a) open up\&down and (b) down\&up. The forbidden regions are denoted by blue color.}
		\label{fig-forbidden}
	\end{figure}
	
	We write $\openud(H)$ and $\opendu(H)$ for the number of point pairs in~$H$ which are open up\&down and down\&up, respectively. 
	Also, we set $\opendiag(H)=\openud(H)+\opendu(H)$.
	
	\begin{claim}\label{claim:basic1b}
		For any Horton set $H$ with $U:=U(H)$ and $L:=L(H)$, the number of basic type 1b $3$-holes in $H$ is equal to
		\[
		\opendiag(U)+\opendiag(L).
		\]
		Moreover, all of these holes are skew.
	\end{claim}
	
	\begin{proof}
		Due to symmetry, it suffices to prove the following two statements:
		\begin{romanenumerate}
			
			\item \label{claim:basic1b:item1}
			the number of type 1b $3$-holes of $H$ with two points in $U$ and a single point in $L$ is equal to $\opendiag(U)$,
			
			\item \label{claim:basic1b:item2}
			all of them are skew.
			
		\end{romanenumerate}
		
		Let $A$ be a basic type 1b $3$-hole in $H$ with two points in $U$ and one point in $L$, and let $\{\alpha\}$ and $\{\beta,\gamma\}$ be its two chains. The point $\alpha$ lies in $U$ and without loss of generality we may assume that $\beta\in U$ and $\gamma\in L$.
		
		Suppose first that $\alpha$ lies to the left of $\beta$. Then $\gamma$ is the rightmost point of $L$, since otherwise the rightmost point of $L$ would lie
		in the interior of the projective $3$-hole $A$. Also, the pair $\{\alpha,\beta\}$ is open up\&down in $U$, since any point of $H$ lying in the forbidden region would lie in the interior of $A$.
		On the other hand, observe that if $\{\alpha,\beta\}$ is open up\&down in $U$
		and $\gamma$ is the rightmost point of $L$ then $A$ is type 1b $3$-hole in $H$.
		Moreover, $A$ is skew, since $\alpha$ lies to the left of $\gamma$.
		Thus, there are $\opendu(U)$ basic type 1b $3$-holes with $\alpha,\beta\in U,\gamma\in L$,
		and $\alpha$ lying to the left of $\beta$. All of them are skew. 
		
		A similar argument as above shows that there are $\openud(U)$ basic type 1b $3$-holes with $\alpha,\beta\in U,\gamma\in L$,
		and $\alpha$ lying to the right of $\beta$, and all of them are skew.
		
		Altogether, there are $\opendu(U)+\opendu(U)=\opendiag(U)$ type 1b $3$-holes with two points in $U$ and one point in $L$. All of them are skew.
		
		Analogously, there are $\opendiag(L)$ type 1b $3$-holes with one point in $U$ and two points in $L$. All of them are skew and the claim follows.
	\end{proof}
	
	\begin{claim}\label{claim:opendiag}
		For every $z\ge1$,
		\[
		\opendiag(CH(z))=2^{2z-2}+2z-1.
		\]
	\end{claim}
	
	\begin{proof}
		First, we compute the contribution of basic pairs (that is, basic two-point subsets) in $CH(z)$ to the estimated quantity. Let $p_0,p_1,\dots,p_{2^z-1}$ be the points of $CH(z)$
		listed in the left-to-right order. Clearly, the pair $\{p_0,p_{2^z-1}\}$ is open both
		up\&down and down\&up. Thus, it contributes two to the computed quantity.
		We claim that any other pair $\{\ell,u\},\ell\in L,u\in U$, contributes exactly one.
		Indeed, if $\ell$ is to the left of $u$ then $\{\ell,u\}$ is open down\&up due to the definition of Horton sets, and it is not open up\&down since either $p_0$ is in the left forbidden region or $p_{2^z-1}$ is in the right forbidden region.
		Similarly, if $\ell$ is to the right of $u$ then $\{\ell,u\}$ is open up\&down, and it is not open down\&up since $p_0$ and $p_{2^z-1}$ lie in the forbidden regions. Thus, the contribution of the basic pairs in $CH(z)$ to the estimated quantity is $2+((2^{z-1})^2-1)= 2^{2z-2}+1$.
		
		Now, suppose that $\{x,y\}$ is a non-basic pair of $CH(z)$. We have $x,y\in L$ or $x,y\in U$. Suppose first that $x,y\in L$. Then $\{x,y\}$ is not open down\&up as the point $p_{2^z-1}\in U$ lies in the right forbidden region. Suppose now that $\{x,y\}$ is open up\&down. Then $x=p_0$
		since otherwise $p_1\in U$ lies in the left forbidden region. Now, let $A$
		be the base of $\{x,y\}=\{p_0,y\}$. Then $x=p_0\in L(A)$, $y\in U(A)$, and
		$y$ is the rightmost point of $U(A)$ since otherwise the rightmost point of $U(A)$
		lies in the right forbidden region. If $x,y$ satisfy the conditions so far then
		$\{x,y\}$ is indeed open up\&down, since $x=p_0$ is the leftmost point 
		of the whole set $CH(z)$, the point $y$ is the rightmost point of $A$ and each point of $CH(z)\setminus A$ is high above $A$ and thus also above the line $xy$.
		
		Since for each $i=1,\dots,z-1$ there is exactly one layer of size $2^i$ containing $p_0$, exactly $z-1$ pairs $\{x,y\}$ in $L$ contribute one to the computed quantity, and all other pairs $\{x,y\}$ in $L$ have no contribution.
		By symmetry, also exactly $z-1$ pairs $\{x,y\}$ in $U$ contribute one to the computed quantity, and all other pairs $\{x,y\}$ in $U$ have no contribution.
		
		Summarizing, the computed quantity equals $(2^{2z-2}+1)+2\cdot(z-1)=2^{2z-2}+2z-1$,
		as claimed.
	\end{proof}

	\begin{proof}[Proof of Claim~\ref{claim:type1b-3-holes}]
		For each $i=1,2,\dots,z$, $CH(z)$ has $2^{z-i}$ layers of size $2^i$. Each of them has
		$2\cdot \left( 2^{2(i-1)-2}+2(i-1)-1 \right) 
		= 2^{2i-3}+4i-6$
		basic type 1b $3$-holes due to Claims~\ref{claim:basic1b} and \ref{claim:opendiag}, and each of the is skew. Thus, there are
		\[
		\sum_{i=2}^z 2^{z-i} \cdot \left( 2^{2i-3}+4i-6 \right)
		= \sum_{i=2}^z 2^{z+i-3} +4i 2^{z-i} -6\cdot 2^{z-i}
		\]
		\[
		=\left( 2^{2z-2} - 2^{z-1} \right) 
		+ \left( 3\cdot 2^{z+1} -4z -8 \right) - \left( 3 2^z - 6 \right)
		\]
		\[
		= 2^{2z-2} + \frac52 2^z -4z -2 
		\]
		type 1b $3$-holes in $CH(z)$.
	\end{proof}

	%%%%%%%%%%%%%%%%%%%%%%%%%%%%%%%%%%%%%%%%%%%%%%%%%%%%%%%%%%%%%
	%%%%%%%%%%%%%%%%%%%%%%%%%%%%%%%%%%%%%%%%%%%%%%%%%%%%%%%%%%%%%
	\section{Proof of Proposition~\ref{proposition:projective5holes}}
	\label{sec:projective5holes_proof}
	%%%%%%%%%%%%%%%%%%%%%%%%%%%%%%%%%%%%%%%%%%%%%%%%%%%%%%%%%%%%%
	%%%%%%%%%%%%%%%%%%%%%%%%%%%%%%%%%%%%%%%%%%%%%%%%%%%%%%%%%%%%%

	We show that every set of at least 7 points contains a projective 5-hole and that the minimum numbers of 5-holes in a set of $n=5,6$ points in the real projective plane in general position satisfy $h_5^{p}(5)=1$ and $h_5^{p}(6)=0$.
	The proof of the first claim is inspired by a proof of Harborth~\cite{Harborth1978}, who showed that every set of 10 points in $\mathbb{R}^2$ contains a 5-hole.
	
	The equality $h_5^{p}(5)=1$ follows simply from the fact any set of five points  from $\RPP$ in general position can be projected to 5 points in convex position in the plane.
	Now, assume that there is a set $P$ of at least 6 points from $\RPP$ in general position that contains no projective 5-hole.
	We will show that then $|P|=6$, obtaining an example of a point set that gives $h_5^{p}(6)=0$.
	
	Since any set of five points  from $\RPP$ in general position can be projected to 5 points in convex position in the plane, there is at least one projective 5-gon in $P$.
	We say that a point $q$ is an \emph{interior point} of a projective 5-gon $Q$ in $P$ if there is a plane $\rho \subset \RPP$ containing $Q$ such that $Q$ is in convex position in $\rho$ and its convex hull contains $q$.
	Let $G$ be a projective 5-gon in~$P$ with the minimum number of interior points among all projective 5-gons in $P$.
	Since $P$ does not contain projective 5-holes, $G$ contains at least one interior point.
	
	If $G$ contains at least two interior points $q_1$ and $q_2$, then there are at least three points of $G$ on one side of the line $\overline{q_1q_2}$.
	These three points together with $q_1$ and $q_2$ determine a projective 5-gon in $P$ with fewer interior points than $G$, which is impossible by the choice of~$G$.
	Thus, $G$ contains exactly one interior point $q$.
	Let $p_0,\dots,p_4$ be the vertices of $G$ traced in the counterclockwise order.
	The point $q$ is then not contained in any triangle spanned by $p_i,p_{i+1},p_{i+2}$ (the indices are taken modulo 5), since otherwise we again obtain a projective 5-gon in $P$ with fewer interior points than $G$; see part~(a) of Figure~\ref{fig:proof_one_projective_5holes}.
	
	For every $i=0,\dots,4$, let $G_i$ be the projective 4-gon with vertices $q,p_i,p_{i+1},p_{i+2}$ (the indices are again taken modulo 5).
	Each $G_i$ is a projective 4-hole in $P$ as $q$ is the only interior point of~$G$.
	Since $P$ does not contain a projective 5-hole, no $G_i$ can be extended to a projective 5-hole in~$P$.
	In particular, this implies that for every $i=0,\dots,4$ no point of $P$ can lie
	\begin{itemize}
		\item 
		to the left of $\overline{p_ip_{i+1}}$ and to the left of $\overline{p_{i+2}q}$,
		\item 
		to the right of $\overline{p_ip_{i+1}}$ and to the right of $\overline{p_{i+2}q}$,
		\item 
		to the left of $\overline{qp_{i}}$ and to the left of $\overline{p_{i+1}p_{i+2}}$,
		\item 
		to the right of $\overline{qp_{i}}$ and to the right of $\overline{p_{i+1}p_{i+2}}$.
	\end{itemize}
	See part~(a) of Figure~\ref{fig:proof_one_projective_5holes} for an illustration of the forbidden regions for $i=0$.
	Altogether, these forbidden regions together with $G$ cover the whole real projective plane; see part~(b) of Figure~\ref{fig:proof_one_projective_5holes}.
	Hence, $P$ cannot contain any further point and we have $|P| = 6$. 
	Moreover, we see that the points $p_0,\dots,p_4,q$ give an example of 6 points in $\RPP$ in general position with no projective 5-hole, proving $h^p_5(6)=0$.
	
	\begin{figure}[htb]
		\centering
		\begin{subfigure}[t]{.5\textwidth}
			\centering
			\includegraphics[width=\textwidth,page=1]{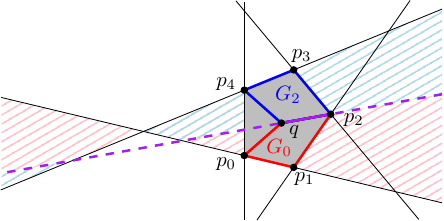}
			\caption{}
			\label{fig:proof_one_projective_5holes_1}
		\end{subfigure}
		\hfill
		\begin{subfigure}[t]{.4\textwidth}
			\centering
			\includegraphics[width=\textwidth,page=2]{figs/proof_one_projective_5holes}
			\caption{}
			\label{fig:proof_one_projective_5holes_2}
		\end{subfigure}
		
		\caption{A sketch of the proof of Proposition~\ref{proposition:projective5holes}. (a) The tiled regions cannot contain further points of~$P$. (b) The forbidden regions cover $\RPP$ outside of $G$. 
		}
		
		\label{fig:proof_one_projective_5holes}
	\end{figure}

	%%%%%%%%%%%%%%%%%%%%%%%%%%%%%%%%%%%%%%%%%%%%%%%%%%%%%%%%%%%%%
	%%%%%%%%%%%%%%%%%%%%%%%%%%%%%%%%%%%%%%%%%%%%%%%%%%%%%%%%%%%%%
	\section{Proofs of Theorems~\ref{thm:construction} and \ref{thm:construction2}}
	\label{sec:two_constructions_proofs}
	%%%%%%%%%%%%%%%%%%%%%%%%%%%%%%%%%%%%%%%%%%%%%%%%%%%%%%%%%%%%%
	%%%%%%%%%%%%%%%%%%%%%%%%%%%%%%%%%%%%%%%%%%%%%%%%%%%%%%%%%%%%%
	
	In this section, we deal with point sets in ${\mathbb{R}}^2\subset\RPP$.
	We explain all the constructions in the plane, and therefore, for the sake of simplicity, we use the notions \emph{affine hole} and \emph{projective hole} for an (affine) hole in the plane and for the intersection of a projective hole with ${\mathbb{R}}^2$, respectively. Thus, a non-affine projective hole is an empty double chain wedge.

	%%%%%%%%%%%%%%%%%%%%%%%%%%%%%%%%%%%%%%%%%%%%%%%%%%%%%%%%%%%%%
	%%%%%%%%%%%%%%%%%%%%%%%%%%%%%%%%%%%%%%%%%%%%%%%%%%%%%%%%%%%%%
	\subsection{Outline of the construction giving Theorems~\ref{thm:construction} and \ref{thm:construction2}}
	\label{sec:two_constructions_outline}
	%%%%%%%%%%%%%%%%%%%%%%%%%%%%%%%%%%%%%%%%%%%%%%%%%%%%%%%%%%%%%
	%%%%%%%%%%%%%%%%%%%%%%%%%%%%%%%%%%%%%%%%%%%%%%%%%%%%%%%%%%%%%

	First, we outline the construction giving both Theorems~\ref{thm:construction} and \ref{thm:construction2}.
	We are given a $k\in\{3,\dots,6\}$ and a positive integer~$n$.
	Our construction uses two integer parameters $a,b\ge2$ satisfying $a\le n^{1/3}$ and $ab\le n$.
	In the proof of Theorem~\ref{thm:construction}, these parameters depend on the value of the parameter $\alpha$ in the theorem.  For the proof of Theorem~\ref{thm:construction2},
	where we are given an integer parameter~$x$,
	we choose
	$a:=2$ and $b\approx \log_2 (x)$. 
	
	Assuming $\sqrt n$ is an integer, we start the construction with the $\sqrt n\times\sqrt n$ integer lattice in the plane, denoted by $L(\sqrt{n}\times \sqrt{n})$,
	and we fix a subset $C_3$ of $\Theta(n^{1/3})$ points in convex position in $L(\sqrt{n}\times \sqrt{n})$.
	We then perturb the lattice to get a so-called \emph{random squared Horton set},
	denoted by $H(\sqrt{n}\times \sqrt{n})$, which is a randomized version~\cite{BaranyValtr2004}
	of the lattice version of so-called 
	Horton sets~\cite{Valtr1992a}, which generalize
	the famous construction of Horton~\cite{Horton1983} of planar point sets in general position with no $7$-holes. The random squared Horton set is described in~\cite[Section~2]{BaranyValtr2004} and denoted by $\Lambda^*$ there.
	
	We consider the $|C_3|$-element subset $C_3^H$ of $H(\sqrt{n}\times \sqrt{n})$
	corresponding to $C_3$. Since $C_3$ is in convex position, the set $C_3^H$
	is also in convex position. We fix an $a$-element subset $C$ of~$C_3^H$, where $a$ is the above mentioned parameter.
	For each $c\in C$, we take a set $S_c$ of $b$ points lying in a very small neighborhood of $c$ and on a unit circle touching the polygon $\conv C_3^H$ in the point $c$. Since the points of $S_c$ are placed very close together on a unit circle, they are almost collinear.
	We consider the set $H(\sqrt{n}\times \sqrt{n})\cap\conv C_3^H$, 
	and denote its union with the sets $S_c,c\in C$, by $T=T(a,b)$; see Figure~\ref{fig:square_construction}.
	The set $T$ has at most $n+ab\le 2n$ points, 
	and it is just a little technicality to adjust its size to~$n$ at the right place in the proof. 
	
	\begin{figure}[htb]
		\centering
		\includegraphics{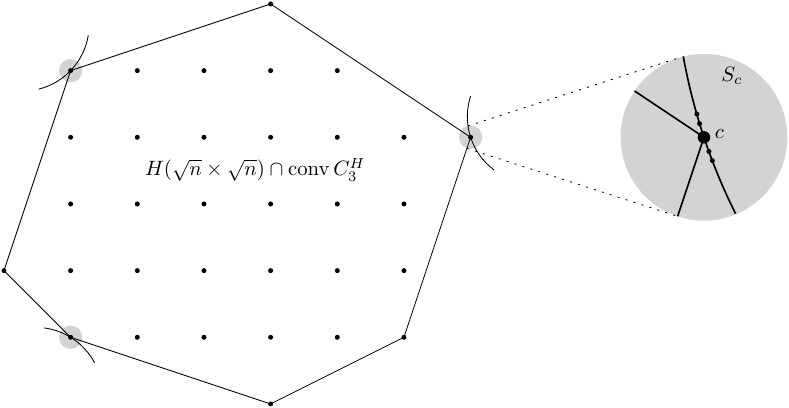} 
		\caption{An illustration of the set~$T(a,b)$ for $a=3$ and $b=5$ (we assume each $c$ lies in $S_c$).}
		\label{fig:square_construction}
	\end{figure}
	
	We now sketch a proof that the set $T$ satisfies Theorems~\ref{thm:construction} and \ref{thm:construction2} for properly chosen parameters $a$ and $b$.
	The random squared Horton set of size $n$ has $O(n^2)$ \affine  holes~\cite{BaranyValtr2004,Valtr1992a}.
	Likewise, using the condition $ab\le n$ and two additional facts, it can be argued that the set $T$ has at most $O(n^2)$ \affine holes that do not lie completely in some $S_c$. The two additional facts are that 
	(i)~the expected number of \affine holes containing a fixed point of $C$ is at most $O(n)$ and 
	(ii)~the expected number of \affine holes containing a fixed pair of points of $C$ is at most $O(n)$.
	The number of \affine $k$-holes that lie completely in one of the sets
	$S_c$ is clearly $a\binom{b}{k}<ab^k$. Thus, the total number of \affine $k$-holes in $T=T(a,b)$ is at most $O(n^2+ab^k)$.
	
	Due to the construction, any $(k-1)$-element subset of any set $S_c$, together with any point of $T\setminus S_c$, forms a \projective $k$-hole. There are $a$ sets $S_c$ and each of them has size $b$. Thus, there are
	at least $a\cdot\binom{b}{k-1}\cdot (|T|-b)=\Theta(ab^{k-1}n)$ \projective $k$-holes in $T$.
	
	Now, Theorem~\ref{thm:construction} is obtained from the above construction 
	by setting the parameters $a,b$ carefully with respect to $\alpha$. Namely, for $\alpha\in[0,\frac{2k-5}3]$ we set
	$a\approx n^{1/3}$ and $b:=n^{(5/3+\alpha)/k}$, and
	for $\alpha\in(\frac{2k-5}3,k-2]$ we set $a\approx n^{1-(1+\alpha)/(k-1)}$ and $ b:=n^{(1+\alpha)/(k-1)}$.
	We remark that in the range $\alpha\in[0,\frac{2k-5}3]$, the parameter $a$ corresponds to its maximum possible size which is the maximum size of a subset in the lattice $L(\sqrt{n}\times \sqrt{n})$ in convex position, and the parameter $b$ grows with~$\alpha$, since increased $\alpha$ allows bigger \affine holes.
	In the range $\alpha\in(\frac{2k-5}3,k-2]$, the parameter $b$ continues to grow with $\alpha$ but $a$ is decreasing to keep the size $ab$ of $S$ below~$n$.
	
	To obtain Theorem~\ref{thm:construction2} from the above construction, we set $a:=2$ and $b\approx\log_2x$.
	Then the number of \affine holes contained in one of the two sets $S_c$ is $\approx a2^b=\Theta(x)$ and the number of other \affine holes in $T$ is again in $O(n^2)$. Any subset of the $(ab=)2b$-element union of the two sets $S_c$ is in convex position or is a double chain, determining a \projective hole.
	Thus, $T=T(2,b)$ has at least $\Theta(2^{2b})=\Theta(x^2)$ \projective holes.
	%Theorem~\ref{thm:construction2} follows.

	%%%%%%%%%%%%%%%%%%%%%%%%%%%%%%%%%%%%%%%%%%%%%%%%%%%%%%%%%%%%%
	\subsection{The parametrized construction in the plane}
	%%%%%%%%%%%%%%%%%%%%%%%%%%%%%%%%%%%%%%%%%%%%%%%%%%%%%%%%%%%%%
	
	Here we describe in detail the (parametrized) construction in the plane which is used with different parameters in the proofs of Theorems~\ref{thm:construction} and \ref{thm:construction2}.
	
	For an integer $t$, $L(t\times t)$ denotes the square lattice
	$L(t\times t):=\{(x,y)\in \mathbb{Z}^2:0\le x,y\le t -1\}$ in the plane. It is well known
	that $L(t\times t)$ contains a subset of size $\Theta(t^{2/3})$ which is in convex position; for example, 
	see \cite{AcketaZunic1993}.
	Let $n\in \mathbb{N}$. We suppose for simplicity that $\sqrt{n}/3$ is an integer.
	We fix a set $C_1\subset L(\sqrt{n}/3\times \sqrt{n}/3)$ of size $\Theta(n^{1/3})$ which is
	in convex position. Let $C_3:=\{(3x,3y):(x,y)\in C_1\}$. The set $C_3$ is a subset of
	$L(\sqrt{n}\times \sqrt{n})$, it is in convex position, and its size is in 
	$\Theta(n^{1/3})$. The coordinates of points of $C_3$ are multiples of $3$ which guarantees the following fact.
	
	\begin{fact}\label{fact:C_3}
		For any three non-collinear points $p,q,r$ of $C_3$, there is a point
		in $ L(\sqrt{n}\times \sqrt{n})$ that lies in the interior of the triangle $pqr$.  
	\end{fact}
	
	\begin{proof}
		The point $\frac13 (p+q+r)$ is a lattice point and it lies inside the triangle $pqr$.
	\end{proof}
	
	We now consider the \emph{random squared Horton set} $H(\sqrt{n}\times \sqrt{n})$
	obtained from $L(\sqrt{n}\times \sqrt{n})$ by a proper perturbation; see~\cite{BaranyValtr2004} for details. We use the following crucial facts about 
	$H(\sqrt{n}\times \sqrt{n})$.
	
	\begin{theorem}[\cite{BaranyValtr2004,Valtr1992a}]\label{thm:randomHset-no7holes}
		The set $H(\sqrt{n}\times \sqrt{n})$ has no  \affine $7$-holes in $\mathbb{R}^2$.
	\end{theorem}
	
	\begin{theorem}[\cite{BaranyValtr2004}]\label{thm:randomHset-numberofholes}
		The expected number of \affine holes in the set $H(\sqrt{n}\times \sqrt{n})$ in $\mathbb{R}^2$ is~$\Theta(n^2)$.
	\end{theorem}
	
	For each point $(i,j)\in L(\sqrt{n}\times \sqrt{n})$, we denote by $h_{ij}$ the counterpart of $(i,j)$ in $H(\sqrt{n}\times \sqrt{n})$ (lying in a small neighborhood of $(i,j)$). We set
	\[
	C_3^H:=\{h_{ij}\in H(\sqrt{n}\times \sqrt{n}): (i,j)\in C_3\}.
	\]
	
	Since $C_3$ is in convex position and $H(\sqrt{n}\times \sqrt{n})$ is a perturbation of $L(\sqrt{n}\times \sqrt{n})$,
	the set $C_3^H$ 
	is also in convex position.
	
	We set 
	\[
	H:=H(\sqrt{n}\times \sqrt{n})\cap \conv C_3^H.
	\]
	
	We further fix a subset $C'$ of $C_3^H$ of size $\lfloor|C_1|/2\rfloor= \lfloor C_3^H/2\rfloor\in\Theta(n^{1/3})$, such that $C'$ contains no pair of consecutive vertices of $\conv C_3^H$.
	For each point $c\in C'$, we fix a disk $D_c$ touching the polygon $\conv C_3^H$
	in the point $c$. Moreover, we choose $D_c$ in such a way that the line touching
	$D_c$ in the point $c$ contains no other point of $C_3^H$.
	
	Let $a,b$ be two positive integers with $a\le |C'|$ and $ab\le n$.
	We fix an arbitrary subset $C$ of~$C'$ of size $a$.
	For each point $c\in C$,
	we fix a set $S_c$ of $b$ points lying on the boundary of $D_c$ in a small neighborhood
	of $c$, whereas $c$ is one of the points of $S_c$.
	We set 
	\[
	S:=\bigcup_{c\in C} S_c \;\;\;\text{ and } \;\;\;
	T:=S\cup H. 
	\]
	Note that $T$ is the disjoint union of the sets $S$ and $H\setminus C$.
	
	The points in each $ S_c $ are chosen so close to $c$ so
	that the following two properties hold.
	
	\begin{enumerate}
		\setcounter{enumi}{0}
		\item \label{construction-item-P1}
		The order type of $H$ does not change if, for each $c\in C$, in $H$ we possibly replace $c$ by any other point of $S_c$.
		
		\item \label{construction-item-P2}
		For any two points $h,h'\in S_c$, all points of the set $T\setminus S_c$ lie on the same side of the line $hh'$.
	\end{enumerate}
	
	\noindent
	We then also have the following properties.
	
	\begin{enumerate}
		\setcounter{enumi}{2}
		\item \label{construction-item-P3}
		If $c\in C$ and $t\in T\setminus S_c$ then $S_c\cup\{t\}$ is a (non-affine) projective hole with the two chains $S_c$ and $\{t\}$.
		
		\item \label{construction-item-P4}
		Every affine hole in $T$ intersects at most two sets $S_c$. 
		This follows from Fact~\ref{fact:C_3}.
		
		\item \label{construction-item-P5}
		If an affine hole in $T$ is not a subset of a set $S_c$ then it intersects $S_c$ in at most two points. Moreover, such two points must be consecutive points of $S_c$ along the (shortest) boundary arc of $D_c$ containing all points of $S_c$. 
		
		\item \label{construction-item-P6}
		Let $X$ be an affine hole of $T$. Let $X'$ be a subset of $H$ obtained from $X$ by replacing, for any $c\in C$, every point in $X\cap(S_c\setminus\{c\})$ by the point $c$ and then removing all multiplicities. Then $X'$ is a hole in $H$.
	\end{enumerate}
	
	We further have:
	
	\begin{enumerate}
		\setcounter{enumi}{6}
		\item \label{construction-item-P7}
		$|T|=\Theta(n)$ and $|C|,|C_1|=\Theta(n^{1/3})$.
	\end{enumerate}
	
	For the proofs of Theorems~\ref{thm:construction} and \ref{thm:construction2} we will also need the following two lemmas.
	
	\begin{lemma}\label{lma:onepoint}
		For any $(a,b)\in L(\sqrt n\times\sqrt n)$, the expected number of holes of $H$ containing the point $h_{ab}$ is $O(n)$.
	\end{lemma}
	\begin{proof}
		
		Suppose that $H(3\sqrt n\times3\sqrt n)$ is constructed in the analogous way as $H(\sqrt n\times\sqrt n)$.
		The set $H(3\sqrt n\times3\sqrt n)$ has at most $O(n^2)$ holes, each having at most $6$ vertices.
		Consider two (not necessarily different) lattice points $(a,b),(c,d)\in L(\sqrt n\times\sqrt n)$.
		Then, due to the construction of random squared Horton sets, the expected number of holes in $H(\sqrt n\times\sqrt n)$ containing the point $h_{ab}$ is smaller or equal to the expected number
		of holes in $H(3\sqrt n\times3\sqrt n)$ containing the point $h_{c+\sqrt n,d+\sqrt n}$.
		The average expected number of holes in $H(3\sqrt n\times3\sqrt n)$ containing the point $h_{c+\sqrt n,d+\sqrt n}$, where we average over all $n$ lattice points $(c,d)\in L(\sqrt n\times\sqrt n)$, is at most $O(n)$, since otherwise $H(3\sqrt n\times3\sqrt n)$ would contain more than $n\cdot O(n)/6=O(n^2)$ holes.
		Thus, also the expected number of holes in $H(\sqrt n\times\sqrt n)$ containing the point $h_{ab}$
		is at most $O(n)$ for each fixed $(a,b)$.
	\end{proof}
	
	\begin{lemma}\label{lma:twopoints}
		Let $c,c'$ be two different points of $C$. Then the expected number of holes in~$H$ containing both $c$ and $c'$ is $O(1)$.
	\end{lemma}
	\begin{proof}
		We apply Lemma~5 from~\cite{BaranyValtr2004} in a similar way in which it is used in~\cite{BaranyValtr2004} for estimating the expected number of empty holes in the random squared Horton set.
	\end{proof}
	
	%Before proving Lemmas~\ref{lma:onepoint} and \ref{lma:twopoints}, we finish the proof of Theorems~\ref{thm:construction} and \ref{thm:construction2}. Lemmas~\ref{lma:onepoint} and \ref{lma:twopoints} are then proved at the end of this section.
	
	Theorems~\ref{thm:construction} and \ref{thm:construction2} will follow quite easily from the following lemma.
	
	\begin{lemma}\label{lma:construction-holes}
		For any $k\ge3$, the expected number of affine $k$-holes in $T$
		is at most $O(n^2+ab^k)$, and $T$ has at least $\Omega(ab^{k-1}n)$ projective holes.
	\end{lemma}
	
	\begin{proof}
		We distinguish the following four types of affine holes $X$ in $T$.
		\begin{itemize}
		    \item \emph{Type I}: $X$ is a subset of one of the sets $S_c$.
		    \item \emph{Type II}: $X$ contains no point of $S=\bigcup_{c\in C}S_c$.
		    \item \emph{Type III}: $X$ contains at least one point of some set $S_c$, at least one point of
		$H\setminus S$, and no point of $S\setminus S_c$.
		    \item \emph{Type IV}: $X$ intersects (exactly) two of the sets $S_c$.
		\end{itemize}
		
		Due to Property~\ref{construction-item-P4}, each affine hole of $X$ is of one of the four types
		I, II, III, and IV.
		Thus, for proving the claimed upper bound on the expected number of affine holes in $T$, it suffices to show that 
		(i) the number of $k$-holes of type I in $T$ is at most $ab^k$ and 
		(ii) the expected numbers of holes of types II--IV in $T$ are at most $O(n^2)$.
		
		The number of $k$-holes of type I in $T$ is $a\cdot \binom{b}{k}<ab^k$.
		
		The expected number of holes of type II in $T$ is in $O(n^2)$ due to Theorem~\ref{thm:randomHset-numberofholes}.
		
		According to Property~\ref{construction-item-P5}, any hole of type III in $T$ intersects one of the sets $S_c$ in one or two consecutive points of $S_c$, and it contains no other points of $S=\bigcup_{c\in C} S_c$.
		There are $a$ sets $S_c$ and each of them has $b$ points and $b-1$ pairs of consecutive points. It follows that holes of type III in $T$ intersect $S$ in at most $a(b+(b-1))=a(2b-1)$ different ways. The expected number of holes of type III intersecting $S$ in one particular way is at most $O(n)$ due to Lemma~\ref{lma:onepoint} and Property~\ref{construction-item-P1}.
		It follows that the expected number of holes of type III is at most $a(2b-1)O(n)=O(abn)\le O(n^2)$.
		
		According to Property~\ref{construction-item-P5}, any hole of type IV in $T$ intersects two of the sets $S_c$ in one or two consecutive points, and it contains no other points of $S=\bigcup_{c\in C} S_c$.
		There are $\binom{a}{2}$ pairs of sets $S_c$. Similarly as above, each set $S_c$ can be intersected in $2b-1$ non-trivial ways by holes of type IV. It follows that holes of type IV in $T$ intersect $S$ in at most $\binom{a}{2}(2b-1)^2$ different ways. The expected number of holes of type IV intersecting $S$ in one particular way is at most $O(1)$ due to Lemma~\ref{lma:twopoints} and Property~\ref{construction-item-P1}.
		It follows that the expected number of holes of type IV is at most
		$\binom{a}{2}(2b-1)^2O(1)=O((ab)^2)\le O(n^2)$.

		It remains to show that $T$ has at least $\Omega(ab^{k-1}n)$ projective holes.
		Any $(k-1)$-element subset $X_1$ of any set $S_c$ forms a projective hole in $T$ together with any point $p\in H\setminus S_c$, where $X_1$ and $\{p\}$ are the chains of the projective hole.
		Thus, there are at least $a\binom{b}{k-1}\cdot (|H|-b)=\Theta(ab^{k-1}n)$ projective holes in $T$.
	\end{proof}

	%%%%%%%%%%%%%%%%%%%%%%%%%%%%%%%%%%%%%%%%%%%%%%%%%%%%%%%%%%%%%
	\subsection{Proof of Theorem~\ref{thm:construction}}
	%%%%%%%%%%%%%%%%%%%%%%%%%%%%%%%%%%%%%%%%%%%%%%%%%%%%%%%%%%%%%
	
	The first part of Theorem~\ref{thm:construction} is the special case $\alpha=0$ in the second part.
	
	Lemma~\ref{lma:construction-holes} immediately gives the second part of Theorem~\ref{thm:construction},
	if the parameters $a$ and $b$ are set as follows.
	For $\alpha\in[0,\frac{2k-5}3]$ we set
	\[
	a:=\min\{|C'|,n^{1/3}\}\in\Theta(n^{1/3}),
	\]
	\[
	b:=n^{(5/3+\alpha)/k}.
	\]
	
	For $\alpha\in(\frac{2k-5}3,k-2]$, we set
	
	\[
	a:=\min\{|C'|,n^{1-(1+\alpha)/(k-1)}\}\in\Theta(n^{1-(1+\alpha)/(k-1)}),
	\]
	\[
	b:=n^{(1+\alpha)/(k-1)}.
	\]

	%%%%%%%%%%%%%%%%%%%%%%%%%%%%%%%%%%%%%%%%%%%%%%%%%%%%%%%%%%%%%
	\subsection{Proof of Theorem~\ref{thm:construction2}}
	%%%%%%%%%%%%%%%%%%%%%%%%%%%%%%%%%%%%%%%%%%%%%%%%%%%%%%%%%%%%%
	
	We apply Lemma~\ref{lma:construction-holes} with the parameters
	$a:=2$ and $b:=\lfloor\log_2x\rfloor$.
	The number of affine holes in $T$ with at most $8$ vertices 
	is at most $O(n^2+\sum_{i=3}^8 b^i)=O(n^2+b^8)$ by
	Lemma~\ref{lma:construction-holes}. Each affine hole with more than $8$ vertices lies completely in one of the two sets $S_c$ forming $S$ due to Theorem~\ref{thm:randomHset-no7holes}
	and Properties~\ref{construction-item-P1}, \ref{construction-item-P4}, and \ref{construction-item-P5}. Thus there are less than $O(2\cdot 2^b)=O(x)$ of them.
	It follows that $T$ has at most $O(n^2+b^8)+O(x)=O(n^2+x)$ holes.
	
	On the other hand, any subset of $S$ of size more than $2$ is a projective hole,
	thus $T$ has at least $2^{2b}-\binom{2b}{2}-2b-1=\Theta(x^2)$ projective holes.

	%%%%%%%%%%%%%%%%%%%%%%%%%%%%%%%%%%%%%%%%%%%%%%%%%%%%%%%%%%%%%
	%%%%%%%%%%%%%%%%%%%%%%%%%%%%%%%%%%%%%%%%%%%%%%%%%%%%%%%%%%%%%%5
	\section{Proof of Proposition~\ref{proposition:manyprojective34holes}}
	\label{sec:more_3holes_and_4holes_proof}
	%%%%%%%%%%%%%%%%%%%%%%%%%%%%%%%%%%%%%%%%%%%%%%%%%%%%%%%%%%%%%
	%%%%%%%%%%%%%%%%%%%%%%%%%%%%%%%%%%%%%%%%%%%%%%%%%%%%%%%%%%%%%

	Here, we prove the following two inequalities about the minimum number of projective 3- and 4-holes in sets of $n$ points in $\RPP$ in general position:
	\[
	h_3^p(P) 
	\ge h_3(P) + \frac{1}{3} \binom{n}{2} \;\;\;\;\;\text{ and }\;\;\;\;\;
	h_4^p(P) 
	\ge h_4(P) + \frac{1}{2} \left(\binom{n}{2}-3n+3\right).
	\]
	To obtain these estimates, we use ideas from~\cite{BaranyFueredi1987}.
	
	Let $P$ be a set of $n$ points in $\mathbb{R}^2 \subset \RPP$ in general position.
	Any two points $p$ and $q$ determine a projective line that can be divided into two projective line segments, exactly one of which is a line segment in $\mathbb{R}^2$.
	We use $pq$ to denote this line segment and we use $\overline{pq}-pq$ to denote the other projective line segment determined by the projective line $\overline{pq}$.
	Let $S$ be the set of all projective line segments $\overline{pq}-pq$ determined by two points $p$ and $q$ of $P$.
	We let $S'$ be the set of all projective line segments from $S$ that are intersected in an interior point by at least one projective line segment from~$S$.
	
	The projective line segments from $S \setminus S'$ and the points of $P$ determine a drawing of a graph in the real projective plane with no two edges crossing.
	By Euler's formula for the real projective plane, this graph has at most $3n-3$ edges and thus $|S\setminus S'| \leq 3n-3$.
	This implies $|S'| \geq \binom{n}{2}-3n+3$.
	
	For every projective line segment $s$ from~$S$,
	consider a plane $\rho_s \subset \RPP$ which fully contains~$s$.
	In $\rho_s$, the two endpoints of $s$ form an \affine 3-hole with the point of $P$ which is closest to~$s$.
	
	For every projective line segment $s'$ from $S'$,
	there is another projective line segment $t'$ from~$S'$ such that $s'$ and $t'$ cross.
	Let $\rho_{s',t'} \subset \RPP$ be a plane
	which fully contains $s'$ and~$t'$.
	It follows from a result by 
	B\'ar\'any and F\"uredi \cite{BaranyFueredi1987} 
	that $s'$ forms a diagonal of an \affine 4-hole in~$\rho_{s',t'}$.
	
	Thus, by Observation~\ref{obs-holes}, the set $P$ contains a \projective 3-hole
	for each $s \in S$
	and a \projective 4-hole for each $s' \in S'$.
	Since we count each such a 3-hole for at most 3 projective line segments from~$S$ and 
	each such a 4-hole for at most 2 projective line segments from~$S'$, 
	$P$ contains at least $\frac{1}{3} \binom{n}{2}$ \projective 3-holes 
	and at least $\frac{1}{2} \left(\binom{n}{2}-3n+3\right)$ \projective 4-holes, which contain a projective line segment from~$S$ (recall that $S' \subset S$).
	Since none of the projective line segments from $S$ is contained in $\mathbb{R}^2$,
	none of these \projective holes is an \affine hole in $\mathbb{R}^2$.
	This completes the proof of Proposition~\ref{proposition:manyprojective34holes}.

	%%%%%%%%%%%%%%%%%%%%%%%%%%%%%%%%%%%%%%%%%%%%%%%%%%%%%%%%%%%%%
	%%%%%%%%%%%%%%%%%%%%%%%%%%%%%%%%%%%%%%%%%%%%%%%%%%%%%%%%%%%%%
	\section{Proof of Theorem~\ref{theorem:random_sets}}
	\label{sec:random_sets}
	%%%%%%%%%%%%%%%%%%%%%%%%%%%%%%%%%%%%%%%%%%%%%%%%%%%%%%%%%%%%%
	%%%%%%%%%%%%%%%%%%%%%%%%%%%%%%%%%%%%%%%%%%%%%%%%%%%%%%%%%%%%%

	Let $K$ be a compact convex subset of $\RPP$ of unit area
	and let $P$ be a set of $n$ points chosen uniformly and independently at random from $K \subset \mathbb{R}^2 \subset \RPP$.
	We show that the expected number of projective 3-holes in $P$ is in $\Theta(n^2)$.
	We first derive the upper bound $O(n^2)$ in Subsection~\ref{sec:random_sets_upper}.
	The quadratic lower bound follows from Proposition~\ref{proposition:manyprojective34holes}.
	In Subsection~\ref{sec:random_sets_lower} we prove that the expected number of \projective holes, which are not \affine holes in $\mathbb{R}^2$, is in $\Theta(n^2)$. 
	
	By Observation~\ref{obs-holes},
	%and the bound from  \cite{BaranyFueredi1987} and \cite{Valtr1995a} on the expected number of \affine 3-holes in~$P$, 
	it suffices to estimate the expected number of 3-holes and double chain 3-wedges that are empty in sets of $n$ points chosen uniformly and independently at random from a convex body $K \subseteq \mathbb{R}^2$ of unit area.
	The expected number of such 3-holes is $2n^2+o(n^2)$ for any planar convex body $K$ by a result of Reitzner and Temesvari~\cite{ReitznerTemesvari2019}.
	Thus, it remains to estimate the expected number of the empty double chain 3-wedges.
	Using a notation from Subsection~\ref{sec:prelim}, 
	for the 3-wedge $W$ of a double chain $A \cup B$ with vertices $A=\{a_1\}$ and $B=\{a_2,a_3\}$, we call the unique point $a_1$ from $A$ 
	the \emph{apex} of~$W$.

	%%%%%%%%%%%%%%%%%%%%%%%%%%%%%%%%%%%%%%%%%%%%%%%%%%%%%%%%%%%%%
	\subsection{The upper bound}
	\label{sec:random_sets_upper}
	%%%%%%%%%%%%%%%%%%%%%%%%%%%%%%%%%%%%%%%%%%%%%%%%%%%%%%%%%%%%%
	
	Let $S$ be a set of $n \geq 4$ points chosen uniformly and independently at random from a convex body $K \subseteq \mathbb{R}^2$ of unit area.
	For a point $p \in S$, we show that the expected number of double chain 3-wedges with apex $p$ that are empty in $S$ is at most $cn$ for some constant $c$.
	
	We label the points of $S\setminus \{p\}$ as $p_1,\ldots,p_{n-1}$
	in the clockwise circular order around $p$.
	For integers $i$ and $j$ with $1 \leq i < j \leq n-1$, let $E_{i,j}$ be the event that the points $p,p_i,p_j$ form a double chain 3-wedge with apex $p$ that is empty in $S$.
	
	\begin{lemma}
		\label{lemma:event_Eij}
		For any $i,j \in \{1,\ldots,n-1\}$ with $i<j$ and $k=\min\{j-i,n-1-(j-i)\}$,
		we have $\Pr[E_{ij}] \le \frac{4}{(k+1)^2}$.
	\end{lemma}
	\begin{proof}
		Let $X$ denote the intersection of $K$ with the open convex cone spanned by the rays $\overrightarrow{pp_i}$ and $\overrightarrow{pp_j}$; see part~(a) of Figure~\ref{fig:projective3holes2}.
		We let $b$ be intersection point of the ray $\overrightarrow{pp_i}$ with the boundary of $K$.
		Similarly, $c$ is the intersection point of the ray $\overrightarrow{pp_j}$ with the boundary of $K$.
		We assume without loss of generality that $p$ is the origin and that the points $b$ and $c$ lie to the right of $p$ and have the same $x$-coordinate.
		Otherwise we apply a suitable volume preserving affine transformation of $\mathbb{R}^2$,
		which does not affect the expected value.
		
		\begin{figure}[htb]
			\centering
			\includegraphics[page=4]{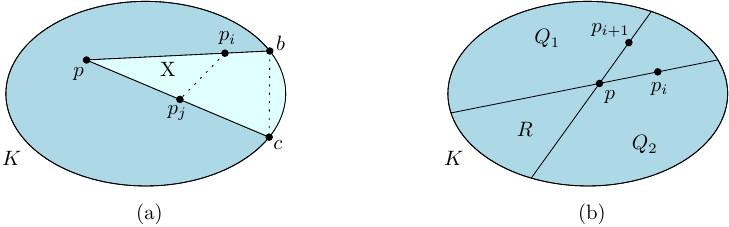}
			\caption{(a) Illustration of the proof of Lemma~\ref{lemma:event_Eij}. (b) Illustration of the proof of the lower bound in Theorem~\ref{theorem:random_sets}.}
			\label{fig:projective3holes2}
		\end{figure}
		
		By the choice of $k$, there are at least $k-1$ points of $S \setminus \{p\}$ in~$X$.
		We denote the set of these points by~$A$.
		Then either $A=\{p_{i+1},\ldots,p_{j-1}\}$ or $A=\{p_{j+1},\ldots,p_{n-1},p_1,\ldots,p_{i-1}\}$.
		We also let $A' = A\cup \{p_i,p_j\}$.

		If the points $p,p_i,p_j$ form a double chain 3-wedge with apex $p$ that is empty in $S$, then all points from $A'$ lie in the triangle $\triangle'$ spanned by $p$, $p_i$, and $p_j$.
		Therefore, we have
		\[
		\Pr[E_{i,j}] 
		\le
		\Pr[A' \subset \triangle']
		\le
		\Pr[A' \subset \triangle' | A' \subset \triangle],
		\]
		where $\triangle$ is the triangle spanned by the points $p$, $b$, and $c$.
		The second inequality follows from the fact that $\triangle' \subseteq \triangle$.
		
		From now on, we assume that all points from $A'$ lie in $\triangle'$.
		Since our distribution is uniform, each of the two points $p_i$ or $p_j$ is the rightmost point of $A'$ with probability $\frac{1}{|A|+2} \leq \frac{1}{k+1}$, as $|A| \geq k-1$. 
		We thus obtain $\Pr[A' \subset \triangle' | A' \subset \triangle] \le 
		\frac{2}{k+1} \cdot \rho$, where
		\[
		\rho = \Pr[A' \subset \triangle' | A' \subset \triangle \text{ and } 
		\text{$p_i$ or $p_j$ is the rightmost point of $A'$}].
		\]
		
		By symmetry, we can assume without loss of generality that $p_i$ is the rightmost point of $A'$.
		Let $x_i$ and $x_j$ be the $x$-coordinates of the points $p_i$ and $p_j$, respectively.
		Note that, since $p_j$ is selected from the triangle, 
		the probability that $p_j$ has $x$-coordinate $x_j$ is precisely $\frac{2}{x_i} x_j$ for every $x_j \in [0,x_i]$,
		and the event $A \subset \Delta'$ happens 
		with probability $(\frac{x_j}{x_i})^{|A|}$.
		Thus, we obtain
		\[
		\rho
		=
		\int_{0}^{x_i}  
		\frac{2}{x_i} \cdot x_j
		\cdot
		\left(\frac{x_j}{x_i}\right)^{|A|}
		{\rm d}x_j
		\le 
		\int_{0}^{x_i}  
		{\frac{2}{x_i} \cdot x_j}
		\cdot
		{\left(\frac{x_j}{x_i}\right)^{k-1}}
		{\rm d}x_j
		= 
		\left[
		\frac{2}{x_i^{k+1}}
		\cdot 
		\frac{x_j^{k+1}}{k+1}
		\right]_{0}^{x_i}
		= \frac{2}{k+1}
		,
		\] 
		where the inequality follows from the fact $|A|\geq k-1$ and $x_j \leq x_i$. 
	\end{proof}
	
	By Lemma~\ref{lemma:event_Eij},
	the expected number of double chain 3-wedges with apex $p$ that are empty in $S$ is
	\[
	\sum_{\substack{i,j \in \{1,\dots,n-1\}\\ i<j}} \Pr[E_{i,j}] 
	\le 
	(n-1) \cdot \sum_{k=1}^{\lfloor \frac{n-2}{2}  \rfloor }\frac{4}{(k+1)^2} 
	\le 
	(n-1) \cdot \sum_{k=1}^{\infty} \frac{4}{k^2} 
	=
	(n-1) \cdot \frac{2\pi^2 }{3}.
	\]
	which gives the upper bound $\frac{2\pi^2 }{3} n(n-1)$ on the expected number of double chain 3-wedges that are empty in $S$.

	%%%%%%%%%%%%%%%%%%%%%%%%%%%%%%%%%%%%%%%%%%%%%%%%%%%%%%%%%%%%%
	\subsection{The lower bound}
	\label{sec:random_sets_lower}
	%%%%%%%%%%%%%%%%%%%%%%%%%%%%%%%%%%%%%%%%%%%%%%%%%%%%%%%%%%%%%

	For each point $p \in S$,
	we first show that the expected number of empty double chain 3-wedges in $S$ with apex $p$ is at least $cn$ for some positive constant $c$.
	Again, we consider the points $p_1,\ldots,p_{n-1}$ of $S\setminus \{p\}$ ordered in the circular clockwise order around $p$.
	
	For some fixed $i \in \{1,\dots,n-1\}$, the lines $\overline{pp_i}$ and $\overline{pp_{i+1}}$ partition $K$ into four regions.
	We denote the region not incident to $p_i$ nor $p_{i+1}$ as $R$, the two regions adjacent to $R$ by $Q_1$ and $Q_2$, and we let $Q=Q_1 \cup Q_2$; see part~(b) of Figure~\ref{fig:projective3holes2}.
	Now, we distinguish three cases depending on the size $|R \cap S|$.
	
	\textbf{Case 1:} 
	Assume $|R \cap S| = 0$.
	Then the points $p$, $p_i$, and $p_{i+1}$ are vertices of an empty double chain 3-wedge with apex $p$, as $p_i$ and $p_{i+1}$ are consecutive in the radial order around $p$.
	
	\textbf{Case 2:}
	Assume $|R \cap S| \ge 2$.
	There are at least two consecutive points $p_j$ and $p_{j+1}$ in $R$.
	Then, the points $p$, $p_j$, and $p_{j+1}$ are vertices of an empty double chain 3-wedge with apex $p$, as $p_i$ and $p_{i+1}$ are also consecutive in the radial order around $p$.
	Note that this double chain 3-wedge is also counted in Case~1 when $p_j$ and $p_{j+1}$ play the role of~$p_i$ and~$p_{i+1}$, respectively.
	
	\textbf{Case 3:}
	Assume $|R \cap S| = 1$.
	We do not find an empty double chain 3-wedge with apex $p$ in this case, but
	we show that this case occurs with probability at most $1/2$ if $n \geq 5$.
	Since the points are uniformly distributed, the probability that a randomly chosen point from $S \setminus \{p\}$ lies in $Q$ equals 
	$\rho = \frac{\area(Q)}{\area(Q \cup R)}$.
		To show $\Pr[|R \cap S|=1] \le 1/2$, observe that
	\[
	\Pr[|R \cap S|=1] 
	=
	(n-3)(1-\rho)\rho^{n-4},
	\]
	as there are $n-3$ points from $S\setminus\{p,p_i,p_{i+1}\}$ and a single one of them lies in $R$ while the remaining $n-4$ points all lie in $Q$.
	Putting $f(x) = (1-x)x^{n-4}$ for $x \in [0,1]$,
	the function $f$ has a unique maximum at $x=1-\frac{1}{n-3}$.
	Moreover, the value $(n-3)f( 1-\frac{1}{n-3}) = (1-\frac{1}{n-3})^{n-4}$
	converges to $\frac{1}{e}$ as $n$ goes to infinity.
	Since $(1-\frac{1}{n-3})^{n-4}$ decreases with increasing $n$ and $(1-\frac{1}{5-3})^{5-4} = \frac{1}{2}$,
	we obtain $\Pr[|R \cap S|=1] \le \frac{1}{2}$ for $n \ge 5$.
	
	\medskip
	
	Since there are $n(n-1)$ possibilities for choosing $p$ and $p_i$, the Cases~1 and~2, which occur with probability at least $1/2$, give the desired quadratic lower bound on the expected number of double chain 3-wedge that are empty in $S$.

	%%%%%%%%%%%%%%%%%%%%%%%%%%%%%%%%%%%%%%%%%%%%%%%%%%%%%%%%%%%%%
	%%%%%%%%%%%%%%%%%%%%%%%%%%%%%%%%%%%%%%%%%%%%%%%%%%%%%%%%%%%%%
	\section{Proof of Theorem~\ref{thm:efficient_counting}}
	\label{sec:counting}
	%%%%%%%%%%%%%%%%%%%%%%%%%%%%%%%%%%%%%%%%%%%%%%%%%%%%%%%%%%%%%
	%%%%%%%%%%%%%%%%%%%%%%%%%%%%%%%%%%%%%%%%%%%%%%%%%%%%%%%%%%%%%

	Let $S$ be a set of $n$ points in the Euclidean plane in general position.
	Mitchell et al.~\cite{MitchellRSW1995} use a dynamic programming approach to determine, for every point $p \in S$, the number of $k$-gons and $k$-holes for $k=3,\ldots,m$, which have $p$ as the bottom-most point.
	The algorithm performs in
	$O(m n^2)$ time and space.
	They also determine the number of $k$-islands in $S$, which have $p$ as the bottom-most point, 
	in $O(m^2 n^3)$ time and space.
	Note that the bottom-most point is unique without loss of generality, as otherwise we perform an affine transformation which does not affect the number of $k$-gons, $k$-holes, and $k$-islands. 
	
	Here, we introduce an algorithm that efficiently computes the number of \projective $k$-gons, $k$-holes, and $k$-islands of a finite set $P$ of $n$ points from $\mathbb{R}^2 \subset \RPP$.
	First, we discuss how to determine the number of \projective $k$-gons in~$P$.
	
	Let $G$ be a \projective $k$-gon with $k\ge 3$
	and let $p_1,p_2$ be two vertices that are consecutive on the boundary of~$G$. 
	If we start at $p_1$
	and trace the boundary of~$G$ in the direction of~$p_2$,
	we obtain a unique cyclic permutation $p_1,\ldots,p_k$ of the vertices of~$G$.
	By starting at $p_2$ and tracing in the direction of $p_1$, we obtain the reversed cyclic permutation.
	It is crucial that, independently from the starting point and the direction, only the $k$ pairs $\{p_i,p_{i+1}\}$ for $i=1,\ldots,k$ (indices modulo~$k$) appear as consecutive vertices along the boundary of~$G$.
	
	For every pair of points $\{s,t\} \in P$, the algorithm will
	count (with multiplicities) the number of \projective $k$-gons in~$P$, which have $s$ and $t$ as consecutive vertices on the boundary.
	Since each \projective $k$-gon is counted exactly $k$ times,
	we can then derive the number \projective $k$-gons in~$P$ by a simple division by~$k$.
	
	For a pair $\{s,t \}$ of distinct points from~$P$, 
	we can choose a line $\ell_{s,t}^+$ ($\ell_{s,t}^-$) which is parallel to the line $\overline{st}$ and lies very close and to the left (right) of $\overline{st}$.
	By removing $\ell_{s,t}^+$ and $\ell_{s,t}^-$, respectively, from $\RPP$, we obtain
	two planes $\rho_{s,t}^+ \subset \RPP$ and $\rho_{s,t}^- \subset \RPP$.
	Now,
	every \projective $k$-gon~$G$ of~$P$, which has $s$ and $t$ as  consecutive vertices on its boundary, is a convex $k$-gon either in $\rho_{s,t}^+$ or in $\rho_{s,t}^-$, but not in both. 
	Note that in both planes~$\rho_{s,t}^+$ and $\rho_{s,t}^-$,   
	$s$ and $t$ lie on the boundary of the convex hull of~$P$.
	Moreover, we can assume that $s$ is the bottom-most point in both planes $\rho_{s,t}^+$ and~$\rho_{s,t}^-$, 
	as otherwise we apply a suitable rotation.
	
	For each of the $\binom{n}{2}$ pairs $\{s,t \}$ of distinct points from~$P$,
	we now count the number of convex $k$-gons in the planes $\rho_{s,t}^+$ and $\rho_{s,t}^-$,
	which have $s$ and $t$ as consecutive vertices on the boundary. 
	This counting can be done in $O(m n^2)$ time and space by using the algorithm of Mitchell et al.~\cite{MitchellRSW1995} with the slight modification that, in the initial phase,  we only count $3$-gons of the form $p_1=s,p_2=t,p_3$; see equation~(3) in~\cite{MitchellRSW1995}.
	Since each \projective $k$-gon $G$ is now counted precisely $k$ times, once for each pair of consecutive vertices along the boundary of~$G$,  this completes the argument for \projective $k$-gons.
	
	Similarly, we count \projective $k$-holes and $k$-islands.
	The time and space requirements of the algorithm from~\cite{MitchellRSW1995}
	for counting \projective $k$-holes, which are incident to the bottom-most point, are the same as for \projective $k$-gons.
	For counting \projective $k$-islands, which are incident to the bottom-most point, the algorithm from \cite{MitchellRSW1995} uses $O(m^2 n^3)$ time and space.

	%%%%%%%%%%%%%%%%%%%%%%%%%%%%%%%%%%%%%%%%%%%%%%%%%%%%%%%%%%%%%
	%%%%%%%%%%%%%%%%%%%%%%%%%%%%%%%%%%%%%%%%%%%%%%%%%%%%%%%%%%%%%
	%Bibliography
	%%%%%%%%%%%%%%%%%%%%%%%%%%%%%%%%%%%%%%%%%%%%%%%%%%%%%%%%%%%%%
	%%%%%%%%%%%%%%%%%%%%%%%%%%%%%%%%%%%%%%%%%%%%%%%%%%%%%%%%%%%%%
	
	%\bibliographystyle{abbrv}
	\bibliographystyle{alphaabbrv-url}
	\bibliography{references}

	%%%%%%%%%%%%%%%%%%%%%%%%%%%%%%%%%%%%%%%%%%%%%%%%%%%%%%%%%%%%%
	%% Appendix
	%%%%%%%%%%%%%%%%%%%%%%%%%%%%%%%%%%%%%%%%%%%%%%%%%%%%%%%%%%%%%
	
	\appendix
	\newpage

\end{document}